\documentclass[12pt]{amsproc}
\usepackage{amssymb,latexsym}
\def\R{{\mathbb R}}
\def\T{{\mathbb T}}
\def\C{{\mathbb C}}
\def\N{{\mathbb N}}
\def\Z{{\mathbb Z}}

\def\CS{{\mathcal S}}

\def\H{{\mathcal H}}

\theoremstyle{plain}
\newtheorem{theorem}{Theorem}[section]
\newtheorem{corollary}[theorem]{Corollary}
\newtheorem{example}[theorem]{Example}

\newtheorem{proposition}[theorem]{Proposition}
\newtheorem{lemma}[theorem]{Lemma}

\theoremstyle{definition}
\newtheorem{definition}[theorem]{Definition}

\theoremstyle{remark}
\newtheorem{remark}[theorem]{Remark}

\begin{document}

\title[Measures on solenoids]
{Probability measures on solenoids corresponding to fractal wavelets}

\author[L.W. Baggett]{Lawrence~W.~Baggett}
\address{Department of Mathematics, Campus Box 395, University of Colorado, Boulder, CO, 80309-0395}
\email{baggett@colorado.edu}
\author[K. D. Merrill]{Kathy~D.~Merrill}
\address{Department of Mathematics, Colorado College, Colorado Springs, CO}
\email{kmerrill@coloradocollege.edu}
\author[J. A. Packer]{Judith~A.~Packer}
\address{Department of Mathematics, Campus Box 395, University of Colorado, Boulder, CO, 80309-0395}
\email{packer@colorado.edu}
\thanks{Research supported in part by a grant from the National Science Foundation DMS--0701913}
\author[A. B. Ramsay]{Arlan~B.~Ramsay}
\address{Department of Mathematics, Campus Box 395, University of Colorado, Boulder, CO, 80309-0395}
\email{ramsay@colorado.edu}


\maketitle 

\begin{abstract}
The measure on generalized solenoids constructed using filters by Dutkay and Jorgensen in \cite{dutjor}  is analyzed further by writing the solenoid as the product of a torus and a Cantor set.
 Using this decomposition, key differences are revealed between solenoid measures associated with classical filters in $\mathbb R^d$ and those associated with filters on inflated fractal sets.  In particular, it is shown that the classical case produces atomic fiber measures, and as a result supports both suitably defined solenoid MSF wavelets and systems of imprimitivity for the corresponding wavelet representation of the generalized Baumslag-Solitar group.  In contrast, the fiber measures for filters on inflated fractal spaces cannot be atomic, and thus can support neither MSF wavelets nor systems of imprimitivity.  
\end{abstract}

\section{Introduction}
Let $d$ be a positive integer and $A$ be a diagonal $d\times d$ matrix whose diagonal entries $N_1,N_2,\cdots, N_d$ are integers greater than 1.  We write ${\bf N}=(N_1,N_2,\cdots,N_d),$ and define $\beta$ to be  the induced map on $\mathbb T^d$ given by $\beta(z)=(z_1^{N_1},z_2^{N_2},\cdots,z_d^{N_d})$.   In this context, the generalized solenoid $\CS_{\beta}$ is the inverse limit of $\mathbb T^d$ under the map $\beta$.  Methods of constructing probability measures on generalized solenoids via filter functions were first explored by Dutkay and Jorgensen in \cite{dutjor} and worked out explicitly for the solenoid given by $d=1, \; N_1=2$ and the filter $m(z)=\frac{1+z^2}{\sqrt 2}$ corresponding to the inflated Cantor set wavelet by Dutkay in \cite{dut}.  In a recent paper, Baggett, Larsen, Packer, Raeburn, and Ramsay used a slightly different construction to obtain probability measures on solenoids from more general filters on $\mathbb T^d$ associated to integer dilation matrices \cite{aijkln}.   The filters used in all of these papers, though not necessarily low-pass in the classical sense of the term \cite{mal},  are required to be non-zero except on a set of measure $0$, and bounded away from $0$ in a neighborhood of the origin.  

In Section \ref{measuredecomp}, we will describe in detail the construction given in \cite{aijkln} of the probability measure $\tau$ on $\CS_{\beta}$.  Using the discussion in Chapters 3 and 4 of P. Jorgensen's book \cite{Jor2}, we then examine the decomposition of $\tau$ over the $d$-torus $\mathbb T^d,$ with fiber measures on the Cantor set.   We show that for the measure on the solenoid built from classical filters in $\mathbb R^d,$  the fiber measures are atomic, while for measures built from filters on inflated fractal sets, the fiber measures have no atoms.  
 
In the remaining two sections, we explore the consequences of this distinction between the classical and fractal cases.  First, in Section \ref{MSF}, we define the notion of a solenoid MSF (minimally supported frequency) wavelet using the Hilbert space $L^2(\CS_{\beta},\tau).$  This definition can be applied to both classical filters and filters on inflated fractals.  Thus we are able to extend the concept of an MSF wavelet to the case of inflated fractal sets, where the standard Fourier transform is not available.  The definition on the solenoid also allows us to compare the classical and fractal cases.  We show that  solenoid MSF wavelets exist if and only if the fiber measures on the Cantor set are almost everywhere atomic.  Thus the difference between the nature of the fiber measures in the classical and fractal cases causes a difference in the existence of solenoid MSF wavelets.  

Section \ref{BS} examines a further implication for the related representation of the generalized Baumslag-Solitar group  on $L^2(\CS_{\beta}, \tau).$    Let ${\mathbb Q_A}=\cup_{j=0}^{\infty}A^{-j}(\mathbb Z^d)\subset \mathbb
Q^d$ represent the $A$-adic rationals in $\mathbb R^d$.  The 
generalized Baumslag-Solitar group $BS_{A}$ is a semidirect product, with
elements in ${\mathbb Q_{A}}\times \mathbb Z$
and with group operation given by 
$$(q_1,m_1)\cdot (q_2,m_2)\;=\;
(q_1+A^{-m_1}(q_2),m_1+m_2),$$
$q_1,q_2\in{\mathbb Q_A},\; m_1,\;m_2\in\mathbb Z.$
In Section \ref{BS}, we show that if $\psi\in L^2(\CS_{\beta}, \tau)$ corresponds to a single wavelet for dilation and translation, then $\psi$ is a solenoid MSF wavelet if and only if the wavelet subspaces $\{W_j: j\in\mathbb Z\}$ for $\psi$ form a system of imprimitivity for this representation, so that it is induced in the sense of Mackey from a representation of the discrete subgroup of $A$-adic rationals $\mathbb Q_{A}.$  Thus a single wavelet on the solenoid is induced from the $A$-adic rationals if and only if the fiber measures are atomic.   In the course of obtaining this result, we generalize a theorem of E. Weber \cite{web} relating ordinary MSF wavelets in $L^2(\mathbb R)$ to wavelet subspaces that are invariant under relevant translation operators.    

The upshot of our study is to draw striking distinctions between solenoids associated with classical and fractal filters.  The key difference is that the multiresolution analyses coming from the standard inflated fractal sets will not carry solenoid MSF wavelets in the sense defined here.  This is not so surprising in light of recent work of Dutkay, D. Larson, and S. Silvestrov \cite{dls}.  In contrast, classical multiresolution analyses carry solenoid MSF wavelets as well as standard Fourier MSF wavelets.   In the classical case, the existence of MSF wavelets was an essential tool in proving  the existence of single wavelets in $L^2(\mathbb R^d)$. It is not yet known whether or not the inflated fractal Hilbert spaces and associated dilation and translation operators support a single wavelet.  Another consequence of the existence of MSF wavelets in the classical case was the fact that the representation of the associated Baumslag-Solitar group is induced from the discrete subgroup $\mathbb Q_{A}.$  We have shown in Section \ref{BS} that the connection between the existence of MSF wavelets and systems of imprimitivity for representations of Baumslag-Solitar carries over to the solenoid spaces.  Thus, because we have solenoid MSF wavelets for solenoids built from classical filters, the representation of Baumslag Solitar on $L^2(\CS_{\beta}, \tau)$ can also be seen to be induced from $\mathbb Q_{A}$ in the classical filter case.  However, for filters on inflated fractals, the representations on the associated solenoid space cannot be induced from the $A$-adic rationals.  


\section{Construction of probability measures on solenoids from filter functions and their direct integral decompositions} 
\label{measuredecomp}

Our context will be as follows:  Let $X\subset \mathbb R^d,$ and suppose $X$ is invariant under under translations by
$\{v: v\in \mathbb Z^d\}$ and under 
multiplication by the expansive integer diagonal matrix $A$ with $a_{k,k}=N_k.$  Define $N\equiv\det A=N_1N_2\cdots N_d.$  Let $\mu$ be a $\sigma$-finite Borel measure on $X$ that is invariant under integer translations.  Suppose
there is a positive constant $K\leq N$ such that $\mu(A(S))=K\mu(S)$ for $S$ a Borel
subset of $X.$  
Define dilation and translation operators $D$ and $\{T_{v}: v\in
\mathbb Z^d\}$ on $L^2(X,\mu)$ by 
$$D(f)(x)=\sqrt{K}f(Ax),\;$$
$$T_{v}(f)(x)=f(x-v),\;f\;\in\;L^2(X,\mu).$$
A calculation shows that $T_{v}D\;=\;DT_{Av}.$

In this paper we will consider two classes of such spaces.  In the classical examples, we have $X=\mathbb R^d$ with Lebesgue measure $\mu$. In the other class of examples, $X$ is built from a fractal $\mathcal F\subset \mathbb R^d$ satisfying  $\mathcal F=\cup_{i=1}^K (A^{-1}\mathcal F+v^i)$ for fixed elements  $v^1,\cdots ,v^K\in\mathbb R^d$.  We assume that this system satisfies the separation requirement given by the open set condition \cite{hutch}.  Following \cite{dutjor42},  $X$ is then defined to be an inflated fractal $$\mathcal R=\cup_{j\in\mathbb Z}\cup_{v\in\mathbb Z^d}A^{-j}(\mathcal F+v),$$ and $\mu$ is defined to be Hausdorff measure of dimension $\log_N(K)$ restricted to $\mathcal R$, which coincides with the Hutchinson measure \cite{hutch}.  It will be useful to describe the classical example in the same terms as the fractal example by taking $\mathcal F$ to be $[0,1]^d,$ $K=N,$ and $\{v^1,v^2,\cdots,v^K\}=\{(\frac{j_1}{N_1},\frac{j_2}{N_2}\cdots\frac{j_d}{N_d}:0\leq j_i<N_i\}.$

In both classes of examples there is a natural multiresolution (MRA) structure on $L^2(X,\mu)$, which can be described as follows.  We define a scaling function  $\phi=\chi_{\mathcal F}$.   Translates of $\phi$ are orthonormal, and we define the core subspace $V_0$ of the MRA to be the closure of their span.  Setting $V_j=D^j(V_0)$ it can be shown (see \cite{aijkln}) that $\cup_{j\in \mathbb Z}V_j$ is dense in $L^2(X,\mu)$ and
$\cap_{j\in\mathbb Z} V_j=\{0\}.$  The inclusion $V_j\subset V_{j+1}$ follows from the fact
that 
\begin{equation}
\label{refine}
\phi = \frac1{\sqrt K}\sum_{i=1}^K DT_{v^i}(\phi).
\end{equation}  
If $z=(z_1,\cdots,z_d)\in \mathbb T^d$ and
$v=\;(v_1,\cdots,v_d)\in\mathbb Z^d,$  we will use the notation $z^v\equiv z_1^{v_1}\cdot z_2^{v_2}\cdots \cdot z_d^{v_d}$.  With this notation, the refinement equation (\ref{refine}) above gives a low pass filter for dilation by $A$
defined by $m(z)=\sum_{i=1}^K\frac1{\sqrt K}z^{v^i},$ 
$z\in\mathbb T^d.$   This filter will satisfy
$$\sum_{\{w: \beta(w)=z\}}|m(w)|^2=N,\; z\in \mathbb T^d.$$

We now recall the definition of the probability measure on the generalized solenoid $\CS_{\beta}$ that was developed in \cite{dutjor}, \cite{dut}, and \cite{aijkln}.   
As above, we write $\beta(z)=(z_1^{N_1},z_2^{N_2},\cdots,z_d^{N_d})$. 
 \begin{proposition}\label{tauexists} (Proposition 6.2 of \cite{aijkln})
Denote by $\pi_n$ the canonical map of $\CS_{\beta}$ onto the $n$th copy of $\mathbb T^d$. Let $m:\mathbb T^d\to\C$ be a Borel function such that the inverse image $m^{-1}(\{0\})$ has Haar measure equal to $0$ that in addition satisfies 
$$\sum_{\{w: \beta(w)=z\}}|m(w)|^2=N,\;a.e.\; z\in \mathbb T^d.$$ 
Then there is a unique probability measure
 $\tau$ on $\CS_{\beta}$ such that for every $f\in C(\mathbb T^d)$,
\begin{equation}\label{easierdef}
\int_{\CS_{\beta}} (f\circ\pi_n)\,d\tau=\int_{\mathbb T^d} f(z)\big(\textstyle{\prod_{j=0}^{n-1}|m(\beta^j(z))|^2}\big)\,dz.
\end{equation}
\end{proposition}
The inverse limit group $\CS_{\beta}$ carries a natural group automorphism induced by the shift, which is commonly studied in topological dynamics:

\begin{corollary}
\label{solenoidauto}
Let $\CS_{\beta}\;=\{(z_n)_{n=0}^{\infty}:\;z_n\in \mathbb T^d,\forall n,\;\beta(z_{n+1})=z_n,\;\forall n\}$ be the inverse limit space described in Proposition \ref{tauexists}.
Define $\sigma:\CS_{\beta}\mapsto \CS_{\beta}$ by $\sigma((z_n)_{n=0}^{\infty})\;=\;(\zeta_n)_{n=0}^{\infty},\;\zeta_n\;=z_{n+1},\;n\; \geq 1.$  Then $\sigma$ is a group automorphism of $\CS_{\beta}$ with inverse given by 
$\sigma^{-1}((w_n)_{n=0}^{\infty})=(z_n)_{n=0}^{\infty},$ where $z_0=\beta(w_0),$ and $z_n\;=\;w_{n-1},\;n\geq 1.$
\end{corollary}
\begin{proof}
It is an easy calculation that $\sigma$ is a group homomorphism, and that its inverse is given by the formula above for $\sigma^{-1}.$  Thus $\sigma$ is a group automorphism, as desired.
\end{proof}

\begin{proposition}
\label{quasiinvariant}
The measure $\tau$ is quasi-invariant with respect to $\sigma.$
that is, $\tau(E) = 0$ if and only if $\tau(\sigma(E)) = 0 $ if and only if $\tau(\sigma^{-1}(E)) = 0.$
The Radon-Nikodym derivatives are given as follows:
\[
\frac{d\tau\circ\sigma}{d\tau}(\eta) = \frac1{|m(\beta(\pi_0(\eta))|^2},
\]
and
\[
\frac{d\tau\circ\sigma^{-1}}{d\tau}(\eta)  = |m(\pi_0(\eta))|^2.
\]
\end{proposition}
\begin{proof}
It suffices to show that
\[
\int_{\CS_\beta} f(\sigma(\eta))\frac1{|m(\beta(\pi_0(\eta)))|^2}\, d\tau(\eta) = \int_{\CS_\beta} f(\eta)\, d\tau(\eta)
\]
and
\[
\int_{\CS_\beta} f(\sigma^{-1}(\eta)) |m(\pi_0(\eta))|^2\, d\tau(\eta) = \int_{\CS_\beta} f(\eta)\, d\tau(\eta).
\]
In fact it is enough to prove these equations for functions
of the form $f\circ\pi_n,$ where $f\in C({\mathbb T^d}).$
We prove the second equality, the first
being proved analogously. \begin{align*}
\int_{\CS_\beta}(f\circ \pi_n)\circ\sigma^{-1}(\eta) |m(\pi_0(\eta)|^2\, d\tau(\eta)
& = \int_{\CS_\beta}  f(\beta(\pi_n(\eta))) |m(\beta^n(\pi_n(\eta))|^2\, d\tau(\eta) \\
& = \int_{\CS_\beta} (f\circ\beta \times |m\circ\beta^n|^2)\circ \pi_n(\eta)\, d\tau(\eta) \\
& = \int_{\mathbb T^d} \frac1{N^n} \sum_{\beta^n(w)=z} f(\beta(w)) |m(\beta^n(w))|^2 \prod_{j=0}^{n-1} |m(\beta^j(w))|^2\, dz \\
& = \int_{\mathbb T^d} \frac1{N^n} \sum_{\beta^n(w)=z} f(\beta(w)) |m(w)|^2 \prod_{j=1}^n |m(\beta^j(w))|^2\, dz \\
& = \sum_{\beta(x)=1} \int_{\mathbb T^d} \frac1{N^{n+1}} \sum_{\beta^n(w)=z} f(\beta(wx)) |m(wx)|^2 \prod_{j=1}^n |m(\beta^j(wx))|^2\, dz \\
& = \int_{\mathbb T^d} \frac1{N^n} \sum_{\beta^n(w)=z} f(\beta(w) )\prod_{j=1}^n |m(\beta^j(w))|^2\, dz \\
& = \int_{\mathbb T^d} \frac1{N^n} \sum_{\beta^n(w)=z} f(w)\prod_{j=0}^{n-1} |m(\beta^j(w))|^2\, dz \\
& = \int_{\CS_\beta} f\circ \pi_n(\eta)\, d\tau(\eta).
\end{align*}
(Note that we made use of a generalization of the alternate form of the definition of $\tau$ given in  \cite[Proposition~4.2(i)]{dut}:
$\int_{\CS_{\beta}} (f\circ\pi_n)\,d\tau=\int_{\T} \frac{1}{N^{n}}\Big(\sum_{\beta^n(w)=z} f(w)\big(\textstyle{\prod_{j=0}^{n-1}|m(\beta^n(w))|^2}\big)\Big)\,dz.$)

\end{proof}

In a broader context where $\beta$ is the multiplicative dual of multiplication by an arbitrary expansive integer dilation matrix $A$, B. Brenken has noted in \cite{brenk} that the corresponding solenoid $\CS_{\beta}$ can be expressed as a locally trivial principal $\Sigma$-bundle over $\mathbb T^d$ for some $0$-dimensional group $\Sigma.$ In Corollary \ref{solenoidauto}, $A$ is the diagonal matrix with diagonal entries $N_1,N_2,\cdots, N_d.$ In the case where $d=1,\;\Sigma$ is a variant of the Cantor set.  
Using another point of view, Jorgensen in \cite{Jor2} views $\CS_{\beta}$ as a bundle over $\mathbb T^d$ whose fibers can be viewed as ``random walks".  Given a filter $m$ on $\T^d,$ for each $z\in \mathbb T^d$ he is able to define a probability measure $P_z$ on the space of trees in the random walks.  
In this section, our aim is to combine the approaches of Brenken and Jorgensen and express the measure $\tau$ on $\CS_{\beta}$ as a direct integral measure over $\T^d$ of Borel measures $\nu_z$ each defined on $\Sigma.$ The measures $\nu_z$ correspond to the random walk measures $P_z$ of Jorgensen in a natural way. 

We first observe that by using exact sequences of compact abelian groups, the fiber $\Sigma$ over $z\in \T^d$ can be describe in the following way:  if 
$$\CS_{\beta}\;=\;\{(z_n)_{n=0}^{\infty}:\;z_n\in\;\mathbb T^d; z_n=\beta(z_{n+1}),n\geq 0\},$$  then fixing $z\in \mathbb T^d,$ the fiber $\pi_0^{-1}(\{z\})$ can be identified with the $0$-dimensional group 
$$\Sigma_{\beta}\;=\;\;\{(z_n)_{n=0}^{\infty}:\;z_n\in\;\mathbb T^d;\;z_n=\beta(z_{n+1}), n\geq 0;\;z_0=1_{\mathbb T^d}\}.$$
Choosing any Borel cross-section $c:\T^d\;\rightarrow \;\CS_{\beta}$ satisfying $\pi_0\circ c\;=\; \text{Id}_{\mathbb T^d},$ 
the correspondence $\rho: \pi_0^{-1}(\{z\})\rightarrow \;\Sigma_{\beta}$ is given by  
$$\rho((z_n)_{n=0}^{\infty})=(z_n)_{n=0}^{\infty}\cdot c(z_0)^{-1}.$$

\begin{example}
\label{1-d}
{\rm We can describe the map $\rho$ explicitly in the case where $d=1$ and $A=(N)$.  We define a Borel cross-section $c:\T\to\CS_{\beta}$ by 
$$c(e^{2\pi i t})\;=\;(e^{\frac{2\pi it}{N^n}})_{n=0}^{\infty},$$
for $t\in[0,1)$ 
The map $\rho: \pi_0^{-1}(\{z\})\rightarrow \;\Sigma_{\beta},$ in this case is given by 
$$\rho((z_n)_{n=0}^{\infty})\;=\;(z_n\cdot e^{-\frac{2\pi it_0}{N^n}})_{n=0}^{\infty},$$
where $z_0=e^{2\pi i t_0}$.
Indeed, in this case, we can go further, and identify $\Sigma_{\beta}$ with the compact abelian group of $N$-adic integers, which is $0$-dimensional, as follows. Let $\Z_N$ denote the $N$-adic integer group, so that 
$$\Z_N\;=\;\prod_{j=0}^{\infty}\{0,1,\cdots,N-1\},$$ with the product topology coming from the discrete topology on the finite set \\$\{0,1,\cdots,N-1\}$ and group action induced by odometer addition, with the ``carrying" operation taking place to the right. The map $\Phi:\prod_{j=0}^{\infty}\{0,1,\cdots,N-1\}=\Z_N\;\to \;\Sigma_{\beta}$ is given by 
$$\Phi((a_n)_{n=0}^{\infty})= (z_n)_{n=0}^{\infty},$$
where
$$z_0\;=\;1\;\text{and } z_n=e^{2\pi i \frac{\sum_{j=0}^{n-1}a_jN^{j}}{N^{n}}}\text{for}\;n>1.$$
By construction, $z_0=1,$ and it is easily checked that $\beta(z_{n})=(z_{n})^N=z_{n-1},\;\forall n\in \N.$
}
\end{example}

Routine calculations show that in the above example, there is a Borel isomorphism $\Theta$ between the Cartesian product
$\T\times \Sigma_{\beta}$ and the $N$-solenoid $\CS_{\beta}$ given by 
$$\Theta:\T\times\Sigma_{\beta}\equiv\mathbb T\times\mathbb Z_N\;\rightarrow\;\CS_{\beta},$$
where $$\Theta(z,(a_n)_{n=0}^{\infty})\;=\;c(z)\cdot (1,e^{2\pi i a_0/N}, e^{2\pi i \frac{a_0 +a_1N}{N^2}},\cdots,e^{2\pi i \frac{\sum_{j=0}^{n-1}a_jN^{j}}{N^{n}}},\cdots).$$ 

To generalize Example  \ref{1-d} to higher dimensions,  we will use the notation\\ $\text{e}(\bold t)\equiv\text{e}(t_1,t_2,\cdots,t_d)\;=\;(e^{2\pi i t_1},e^{2\pi i t_2},\cdots, e^{2\pi i t_d}). $
Write $\Sigma_{\beta}$ for the kernel of the projection map $\pi_0:\CS_{\beta}\to \mathbb T^d,$ and define the cross section $c:\mathbb T^d\rightarrow\CS_{\beta}$ by $$c(e(\bold t))=(e(A^{-n}\bold t))_{n=0}^{\infty}.$$
There is a Borel isomorphism $\Theta:\T^d\times\Sigma_{\beta}\;\rightarrow\;\CS_{\beta}$ defined by 
$$\Theta(z,\eta)=c(z)\eta$$
Under the Borel isomorphism $\Theta,$ the shift map 
$\sigma: \CS_{\beta}\mapsto \CS_{\beta}$  corresponds to a map on the Cartesian product $\T^d\times \Sigma_{\beta}$ which we denote by
$$\widetilde{\sigma}=\Theta^{-1}\circ\sigma\circ \Theta:\T^d\times \Sigma_{\beta}\;\to\;\T^d\times \Sigma_{\beta}.$$

A computation then shows that 
\begin{equation}
\label{deftildesigma}
\widetilde{\sigma}(\text{e}({\bf t}),(\eta_n)_{n=0}^{\infty}))=
(\text{e}(A^{-1}({\bf t}))\eta_1,([\text{e}(A^{-(n+1)}({\bf t}))\eta_{n+1}]\cdot 
[c(\text{e}(A^{-1}({\bf t})\eta_1)_n]^{-1})_{n=0}^{\infty}), 
\end{equation}
and
\begin{equation}
\label{deftildesigmainverse}
\widetilde{\sigma}^{-1}(\text{e}({\bf t}),(\eta_n)_{n=0}^{\infty})
=(\text{e}({\bf t}), (1, (\text{e}(A^{n-1}({\bf t}))\cdot \eta_{n-1}[c(\text{e}(A({\bf t})))_n]^{-1})_{n=1}^{\infty}).
\end{equation}
where $\eta_0=1$ and $c:\T^d\to\CS_{\beta}$ is the Borel cross-section discussed above.

We  let 
$$\Z^d_{\bf A}\;=\;\prod_{j=0}^{\infty}[\{0,1,\cdots,N_1-1\}\times \{0,1,\cdots,N_2-1\}\times\cdots \times\{0,1,\cdots,N_d-1\}]\subset \prod_{j=0}^{\infty}[\Z^d].$$ 
Just as in the above example, we can identify $\Sigma_{\beta}$ with $\Z^d_{\bf A}.$  It is not difficult to show that  the map $\Theta$ then can be described by 
$$\Theta(z,({\bf a}_n)_{n=0}^{\infty})\;=\;c(z)\cdot (1,\text{e}(A^{-1}({\bf a}_0)), \text{e}(A^{-2}({\bf a}_0 +A({\bf a}_1))),\cdots,
\text{e}(A^{-n}(\sum_{j=0}^{n-1}A^j({\bf a}_j))),\cdots).$$ 
while $\widetilde{\sigma}$ can be described by
\begin{equation}
\label{deftildesigmafunny}
\widetilde{\sigma}(\text{e}({\bf t}), ({\bf a}_j)_{j=0}^{\infty})\;=\;(\text{e}(A^{-1}({\bf t}+{\bf a}_0)),({\bf a}_{j+1})_{j=0}^{\infty}).
\end{equation}
Similarly, if we define $s:[0,1)^d\to \mathbb Z_A$ by $s({\bf t})={\bf k}=(k_1,k_2,\cdots, k_d)$ if ${\bf t}\in [\frac{k_1}{N_1},\frac{k_1+1}{N_1})\times [\frac{k_2}{N_2},\frac{k_2+1}{N_2})\times \cdots \times [\frac{k_d}{N_d},\frac{k_d+1}{N_d}) ,$ then $\widetilde{\sigma}^{-1}$ is giving by 
\begin{equation}
\label{deftildesigmainvfunny}
\widetilde{\sigma}^{-1}(\text{e}({\bf t}), ({\bf a}_j)_{j=0}^{\infty})\;=\;(\text{e}(A({\bf t})), (s({\bf t}),{\bf a}_0,{\bf a}_1,\cdots)).
\end{equation}

Using the map $\Theta,$ we obtain the following result concerning the decomposition of the measure $\tau$ of Proposition \ref{tauexists}.  We remark here that while we have been able to make an explicit decomposition, of the measure, there is a general theorem
that could have been applied.  References to earlier work in this direction and a proof of a suitable
general theorem can be found in Lemma 4.4 of \cite{ef}.
\begin{proposition}
\label{taudecomp}
Let $m:\T^d\to\C$ be a Borel function such that the inverse image $m^{-1}(\{0\})$ has Haar measure equal to $0$ that in addition satisfies 
$$\sum_{\{w\in\T^d: \beta(w)=z\}}|m(w)|^2=N,\;a.e. \;z\in \T^d.$$ Let $\tau$ be the Borel measure on $\CS_{\beta}$ constructed using the filter $m.$   Using the Borel isomorphism $\Theta:\T^d\times \Z^d_{A}\;\rightarrow\;\CS_{\beta}$ defined in the above paragraph, the measure $\widetilde{\tau}=\tau\circ\Theta$ corresponds to a direct integral 
$$\widetilde{\tau} = \int_{\T^d}d\nu_z dz,$$ where the measure $\nu_z$ for $z=e(\bold t)$ on $\pi_0^{-1}(\{z\})\cong \Sigma_{\beta}\cong\mathbb Z^d_A$ is defined on cylinder sets by 
$$\nu_z(\{{\bf a}_0\}\times\cdots \times \{{\bf a}_{k-1}\}\times \prod_{j=k}^{\infty}[\{0,1,\cdots,N_1-1\}\times \{0,1,\cdots,N_2-1\}\times\cdots \times\{0,1,\cdots,N_d-1\}])$$
$$=\;\frac{1}{N^k}\prod_{j=1}^{k}|m(\text{e}(A^{-j}({\bf t})\cdot \text{e}(A^{-j}(\sum_{i=0}^{j-1}A^i({\bf a}_i)))|^2.$$ 
\end{proposition} 
\begin{proof}
We have already seen how the measure $\tau$ is defined on the solenoid $\CS_{\beta}$ as follows:  If $f\in C(\T^d)$ and $k\in \N$ are fixed, then
$$\int_{\CS_{\beta}}f\circ \pi_k((z_i)_{i=0}^{\infty})d\tau((z_i))\;=\;\int_{\T^d}\frac{1}{N^k}\sum_{\{w\in \T^d: \beta^k(w)=z\}}f(w)\prod_{j=0}^{k-1} |m(\beta^j(w))|^2dz,$$
where for fixed $k\in\N,\;\pi_k:\CS_{\beta}\rightarrow \T^d$ is the projection map described earlier.  
Then for $f\in C(\T^d),$
$$\int_{\T^d}\int f\circ \pi_k\circ \Theta (z, ({\bf a}_i)_{i=0}^{\infty})d\nu_z dz$$
$$=\;\int_{\T^d}\int f\circ \pi_k(c(z)\cdot (1,\text{e}(A^{-1}({\bf a}_0), \text{e}(A^{-2}({\bf a}_0+A({\bf a}_1))),\cdots, \text{e}(A^{-k}(\sum_{j=0}^{k-1}A^j{\bf a}_j)),\cdots)d\nu_z dz$$
$$=\; \int_{[0,1)^d}\left[\int f(\text{e}(A^{-k}({\bf t}))\text{e}(A^{-k}(\sum_{j=0}^{k-1}A^j{\bf a}_j))d\nu_z\right]d{\bf t}$$
\begin{multline*}
=\int_{[0,1)^d} \frac{1}{N^k}\sum_{{\bf a}_0,{\bf a}_1,\cdots,{\bf a}_{k-1}\in \prod_{i=1}^d[\{0,1\cdots,N_i-1\}]}f(\text{e}(A^{-k}({\bf t}))\text{e}(A^{-k}(\sum_{j=0}^{k-1}A^j({\bf a}_j)))\times \\
\times \prod_{j=1}^{k}
|m(\text{e}(A^{-j}({\bf t}))\text{e}(A^{-j}(\sum_{i=0}^{k-1}A^i({\bf a}_i)))|^2d{\bf t}
\end{multline*}
$$=\;\int_{\T^d}\frac{1}{N^k}\sum_{\{w\in \T^d: \beta^k(w)=z\}}f(w)\prod_{j'=0}^{k-1}|m(\beta^{j'}w)|^2dz$$
$$=\;\int_{\CS_{\beta}}f\circ \pi_k((z_i)_{i=0}^{\infty})d\nu((z_i)).$$
It follows that with respect to the Borel isomorphism $\Theta:\T^d\times \mathbb Z^d_A\rightarrow\;\CS_{\beta},$ we have 
$$\int d\nu_z dz=\widetilde{\tau}.$$
\end{proof}

It follows directly from Proposition \ref{quasiinvariant} and the definition of $\Theta$ that
the measures $\widetilde\tau$
and $\widetilde\tau\circ \widetilde\sigma^{-1}$ are equivalent,
and that the Radon-Nikodym derivative is given by
$$\frac{d\widetilde{\tau}\circ \widetilde{\sigma}^{-1}}{d\widetilde{\tau}}=|m(z)|^2.$$

We now investigate conditions on the filter $m:\T^d\rightarrow \C$ under which any of the measures $\nu_z:\;z\in \T^d$ will be atomic.  
This has been investigated within the framework of random walks in Chapters 2 and 3 of \cite{Jor2}; here, we take a slightly different approach.

\begin{lemma}  
\label{lemmaatoms}
Let $m:\T^d\to \C$ be a filter for dilation by the matrix $A$ as defined in Proposition \ref{taudecomp}.  For each $z\in\T^d,$ let $\nu_z$ be the measure on $\Z^d_{\bf A}$ constructed in Proposition \ref{taudecomp}.  Then $\nu_z$ has atoms if and only if there is a sequence $({\bf a}_j)_{j=0}^{\infty}\in \Z^d_{\bf A}$ such that for each $n\in \N,$ 
$m(\text{e}(A^{-n}({\bf t}))\cdot \text{e}(A^{-n}(\sum_{i=0}^{n-1}A^i({\bf a}_i))))\not=0,$ and moreover 
such that the sequence
$$\frac{1}{\sqrt{N}}|m(\text{e}(A^{-n}({\bf t}))\cdot \text{e}(A^{-n}(\sum_{i=0}^{n-1}A^i({\bf a}_i))))|$$ converges to 1 at a sufficiently rapid rate. 
\end{lemma}
\begin{proof}
Suppose $z\in \T^d$ and $({\bf a}_j)_{j=0}^{\infty}\in \Z^d_{\bf A}$ satisfy 
$\nu_z(\{({\bf a}_j)_{j=0}^{\infty}\})>0.$
Since $({\bf a}_j)_{j=0}^{\infty}$ is the unique element of the nested intersection 
$$\cap_{k=0}^{\infty}[\{{\bf a}_0\}\times\cdots \times \{{\bf a}_k\}\times \prod_{j=k+1}^{\infty}[\prod_{i=1}^d (\{0,1,\cdots,N_i-1\})]_j],$$ it follows that writing $z=\text{e}({\bf t})$ with ${\bf t}\in [0,1)^d,$ we have 
$$\nu_z(\{({\bf a}_j)_{j=0}^{\infty}\})=\;\lim_{k\to\infty}\nu_z(\{{\bf a}_0\}\times\cdots \times \{{\bf a}_{k-1}\}\times \prod_{j=k}^{\infty}[\prod_{i=1}^d (\{0,1,\cdots,N_i-1\})]_j)$$
$$=\;\lim_{k\to\infty}\frac{1}{N^k}\prod_{j=1}^{k}|m(\text{e}(A^{-j}({\bf t}))\cdot \text{e}(A^{-j}(\sum_{i=0}^{j-1}A^i({\bf a}_i))))|^2.$$ Thus in order to have  $\nu_z(\{({\bf a}_j)_{j=0}^{\infty}\})>0$ it is necessary and sufficient that the infinite product
$$\prod_{j=1}^{\infty}\frac{1}{\sqrt{N}}|m(\text{e}(A^{-j}({\bf t}))\cdot \text{e}(A^{-j}(\sum_{i=0}^{j-1}A^i({\bf a}_i))))|$$ converge to a positive number.
This will happen if and only if 
$$m(\text{e}(A^{-j}({\bf t}))\cdot \text{e}(A^{-j}(\sum_{i=0}^{j-1}A^i({\bf a}_i))))\not=0\;\forall j\in\mathbb N$$ and the terms $\frac{1}{\sqrt{N}}|m(\text{e}(A^{-j}({\bf t}))\cdot \text{e}(A^{-j}(\sum_{i=0}^{j-1}A^i({\bf a}_i))))|$ converge to $1$ at a sufficiently rapid rate as $j\;\to\;\infty.$
\end{proof}

\begin{example}
\label{dynamicsex}
{\rm
We discuss the conditions of Lemma \ref{lemmaatoms} further in the case where $d=1,\; A=(N),$ and $m$ is low-pass in the classical sense of \cite{mal}. Thus, suppose that $m(1)=\sqrt{N}$, $m$ is Lipschitz continuous in a neighborhood of $z=1,$ and $m$ satisfies Cohen's condition, so that it is non-zero in a large enough neighborhood of $1.$ Let us also assume that $m$ is a Laurent polynomial in $z\in\T,$ (i.e., is a so-called ``FIR" filter). The discussion which follows has its roots in the description of ``cycles" given by Dutkay and Jorgensen in Section 2 of \cite{dutjorJMP}.

 Let ${\mathcal Z}_m\;=\;\{ z\in\mathbb T: |m(z)|\;=\;\sqrt{N}\}.$  Since $m$ is a Laurent polynomial, so is $|m|^2,$ and it follows that the set ${\mathcal Z}_m$ will be finite.    Following Lemma \ref{lemmaatoms}, suppose $( a_n)_{n=0}^{\infty}\in \prod_{i=0}^{\infty}\{0,1,\cdots,N-1\}$ has the property that for some $t\in[0,1)$, $m(\text{e}(\frac t{N^j})\cdot \text{e}(\frac{\sum_{i=0}^{j-1}{a}_iN^{i}}{N^j}))\not=0\;\forall j\in\mathbb N$.  In this case  it is possible that for $z=e(t)$,
$$\nu_z(\{({a}_j)_{n=0}^{\infty}\})\;>\;0,$$  so that 
$$\nu_z(\{({a}_j)_{n=0}^{\infty}\})\;=\;\prod_{j=1}^{\infty}\frac{1}{N}|m(\text{e}({t}/N^j)\cdot \text{e}(\frac{\sum_{i=0}^{j-1}{a}_iN^{i}}{N^j}))|^2.$$

Since $\lim_{j\to\infty}\text{e}({t}/N^j)=1,$ by continuity we have 
$$\lim_{j\to\infty}\frac{1}{\sqrt{N}}|m(\text{e}({t}/N^j)\cdot \text{e}(\frac{\sum_{i=0}^{j-1}{a}_iN^{i}}{N^j}))|=\lim_{j\to\infty}\frac{1}{\sqrt{N}}|m(\text{e}(\frac{\sum_{i=0}^{j-1}{a}_iN^{i}}{N^j}))|.$$

But in order for $\lim_{j\to\infty}|m(\text{e}(\frac{\sum_{i=0}^{j-1}{a}_iN^{i}}{N^j}))|^2\;=\;N,$ the values $\{\text{e}(\frac{\sum_{i=0}^{j-1}{a}_iN^{i}}{N^j})\}$ must be becoming arbitrarily close to elements in ${\mathcal Z}_m$.  Given the fact that $\#\;{\mathcal Z}_m$ is finite, and $m$ is a Laurent polynomial and continuous on the unit circle, this
dictates restrictive conditions on the choice of the $({a}_i)_{i=0}^{\infty}$.  In exploring these conditions, we will make use of the key observation that 
$$\frac{\sum_{i=0}^{k-1}a_iN^i}{N^k}\;=\;\sum_{i=1}^{k}\frac{ a_{k-i}}{N^i}.$$

One way the conditions can be met is if $\{e(\frac{\sum_{i=0}^{j-1}{a}_iN^{i}}{N^j})\}$ converges to a point $z_0$ in ${\mathcal Z}_m,$ which happens in the classical case for $z_0= 1 \in {\mathcal Z}_m.$ 
If $\lim_{k\to\infty}\text{e}( \frac{\sum_{i=0}^{k-1}{a}_iN^i}{N^k})$ exists, the sequence is Cauchy, and an easy argument shows that there must exist $J>0$ and ${a}\;\in\;\{0, 1, \cdots, N-1\}$ such that for all $j>J,\;{a}_j={a}.$ 
 
A second way to achieve $\lim_{j\to\infty}|m(\text{e}(\frac{\sum_{i=0}^{j-1}{a}_iN^{i}}{N^j}))|^2=N$ is  to have the points $e(\frac{\sum_{i=0}^{k-1}{a}_iN^{i}}{N^k})$  bounce around very close to a variety of points ${z}\in {\mathcal Z}_m$ as $k$ gets larger and larger.  
  When ${\mathcal Z}_m$ is finite, this can only happen under special circumstances.  Suppose that for some fixed positive integer $l\;\geq\; 1,$ and $\{d_1,d_2,\cdots, d_l\}\in \{0,1,\cdots, N-1\},$ we have the repeating $N$-adic expression
  $$x_k= \sum_{j=0}^{\infty}\sum_{i=1}^l\frac{d_{i+k}}{N^{jl+i}}$$
  satisfying $e(x_k)\in{\mathcal Z}_m .$ 
 for $k=0,1,\cdots,l-1$.  Then it is possible to  construct  $({a}_i)_{i=0}^{\infty}\in \prod_{i=0}^{\infty}\{0,1,\cdots,N-1\}$ that is an atom as follows: 
  
 Fix a non-negative integer $J,$ and choose ${a}_0,\;{a}_1,\;{a}_2,\cdots, {a}_{Jl}\in\{0,1,\cdots, N-1\}$ arbitrarily, and for all $j\geq 0$ and $i=1,\cdots, l$, let 
$${a}_{(J+j)l+i}=d_{l-i+1}$$
It will then be the case that as $k\to\infty,$
$$\text{e}\left(\frac{\sum_{i=0}^{k-1}a_iN^i}{N^k}\right )\;=\;\text{e}(\sum_{i=1}^{k}\frac{a_{k-i}}{N^i})$$ will tend to one of the 
$l$ values $e({x}_0),e({x}_1),\;\cdots,e(x_{l-1})\in{\mathcal Z}_m.$

It might seem that any choice 
$\{{d}_1,{d}_2,\cdots,\;{d}_l\}$ will do in this construction, for $l$ as large as we want.  But this is not the case, because as noted above, the accumulation points $\text{e}({x}_0),\text{e}({x}_1),\cdots, \text{e}({x}_{l-1})$ all need to be in ${\mathcal Z}_m.$ Moreover if one has made such a  choice of ${d}_i$'s, with
$\text{e}({x}_0)=\text{e}(\sum_{j=0}^{\infty}\sum_{i=1}^l\frac{d_{i}}{N^{jl+i}}),$  then
$$\text{e}({x}_0)=\text{e}(\sum_{i=1}^l\frac{{d}_i}{N^i}\sum_{j=0}^{\infty}\frac{1}{(N^l)^j})$$
$$\;=\;\text{e}(\sum_{i=1}^l\frac{{d}_i}{N^i}\frac{N^l}{N^l-1})=\;\text{e}(\frac{\sum_{i=1}^l{d}_i\cdot N^{l-i}}{N^l-1})$$
$$=\;e^{2\pi i \frac{\sum_{i=0}^{l-1}{d}_{l-i}\cdot N^{i}}{N^l-1}}.$$
Thus the only possible elements of ${\mathcal Z}_m$ that can possibly give rise to sequences having atoms are numbers modulus $1$ of the form $\text{e}({r})$ where ${r}$ is a rational number in $[0,1)$ lying in 
$$\cup_{l=1}^{\infty}\{{r}=\frac{p}{q};\;q=N^{l}-1\;\text{and}\;0\leq\;p\;<N^l-l\}.$$

We leave to the interested reader the verification that if $d>1$ with dilation factors $N_1,\;N_2,\;\cdots,\;N_d,$
the only possible elements of ${\mathbb Z}_m$ that can possibly give rise to sequences having atoms are numbers modulus $1$ of the form $\text{e}(\bf{r})$ where $\bf{r}$ is a $d$-tuple of rational numbers in $[0,1)^d$ that are elements of 
$$\cup_{i=1}^d\cup_{l_i=1}^{\infty}\{{r}=(\frac{p_1}{q_1},\frac{p_2}{q_2}, \cdots, \frac{p_d}{q_d});\;q_i=N^{l_i}-1\;\text{and}\;0\leq\;p_i\;<N^{l_i}-l\}.$$
}
\end{example}

More generally, it is intriguing that the existence of atoms is independent of the fiber $z=\text{e}({\bf t})\;\in\;\T^d$ chosen.   However, if an atom does exist in the fiber, the value of the fiber measure $\nu_z$ on the atom will depend on $z.$

We expand further on the decomposition of the measure $\tau$ in the case of a classical low-pass filter, and recall Theorem 6.6 of \cite{aijkln}. Reviewing the set-up for this theorem, let $m:\T^d\to \C$ be a low-pass filter for dilation by $N$ in the sense of \cite{mal}, that is, suppose that $m$ is Lipschitz continuous at $1,$ is non-zero in a sufficiently large neighborhood of $1,$ and satisfies
$$\sum_{w: \beta(w)=z}|m(w)|^2=N,\;a.e. z\in \T^d.$$  We also assume $m$ be a Laurent polynomial in the variables $(z_1,z_2,\cdots z_d)$ of $z.$ Let $A$ be the diagonal matrix with diagonal entries $N_1, N_2,\cdots, N_d,$ and let $N=\prod_{i=1}^d N_i.$  It is then a standard construction due to Mallat and Meyer that defining the Fourier transform of the scaling function $\phi\;\in\; L^2(\R^d)$ by 
$$\widehat{\phi}({\bf t})\;=\;\prod_{j=1}^{\infty}\frac{1}{\sqrt{N}}m(\text{e}(A^{-j}({\bf t})),$$ this product converges pointwise and in $L^2(\R^d).$ Theorem 6.6 of \cite{aijkln} states that if we define the winding line map 
$w:\R^d\to\CS_{\beta}$ by $$w({\bf t})\;=\;(\text{e}({\bf t}),\text{e}(A^{-1}({\bf t})),\;\text{e}(A^{-2}({\bf t})),\cdots, \text{e}(A^{-n}({\bf t})),\cdots, ),$$
then for any $f\in L^2(\CS_{\beta},\tau),$
$$\int_{\CS_{\beta}}f\;d\tau\;=\;\int_{\mathbb R^d}f\circ w({\bf t})|\widehat{\phi}({\bf t})|^2dt.$$
This result states that $\tau$ is supported on the image of $w$ in $\CS_{\beta},$ that is, on the winding line $L_w,$ the image of $\R$ under $w$ in the $N$-solenoid.  In this case, the structure of the fiber measures $\nu_z$ can be deduced exactly.   Let $O_m\;=\;\{x\in [0,1): |m(e^{2\pi ix})|=\sqrt{N}\}$  and $F_m\;=\{(a_n)_{n=0}^{\infty}: \;\sum_{i=0}^{\infty}\frac{a_i}{N^{i+1}}\;=\;x\in\;O_m\}$ be as before. Defining $z=e^{2\pi i t}$ with $t\in [-\frac{1}{2},\frac{1}{2}),$ if we fix $(a_n)_{n=0}^{\infty}\in F_m$ and make the assumption that $m(e^{2\pi i t/N^n}\cdot e^{2\pi i \frac{\sum_{i=0}^{n-1}a_iN^{i}}{N^n}})\not=0\;\forall n\in\mathbb N,$ we then obtain 
$$\nu_z(\{(a_n)_{n=0}^{\infty}\})\;=\;\prod_{j=1}^{\infty}\frac{1}{N}|m(e^{2\pi i t/N^j}\cdot e^{2\pi i \frac{\sum_{i=0}^{j-1}a_iN^{i}}{N^j}})|^2\;>\;0.$$
As a result of this analysis, we have deduced the following result:
\begin{proposition}
\label{decompclas}
Let $m:\T^d\to\;\C$ be a classical low-pass filter, and let $\tau$ be the Borel probability measure on the solenoid $\CS_{\beta}$ associated to $m$ by Proposition \ref{tauexists}.  Decompose $\tau$ as $d\tau=\int_{\T^d} d\nu_zdz$ as in Proposition \ref{taudecomp}.  Then for almost every $z\in \T^d,$ there is a countable subset $E_z\subset \pi_0^{-1}(\{z\})$ such that $\nu_z(\omega)>0\;\forall \omega\in E_z$ and  $\nu_z(\pi_0^{-1}(\{z\})\backslash E_z)=0.$
\end{proposition}
\begin{proof}
We know $\nu$ is supported on $L_w=\;\{w(t):t\in \R\}\subset \CS_{\beta}.$  Because of this, 
$\tau(\CS_{\beta}\backslash L_w)=0.$  It follows that for almost all $z\in \T,\;\nu_z(\pi_0^{-1}(\{z\})\backslash L_w)=0.$  Thus for almost all $z\in\T,$ the nature of the measure $\nu_z$ on $\pi_0^{-1}(\{z\})$ can be deduced from its behavior on the intersection $\pi_0^{-1}(\{z\})\cap E_w.$
Now an arbitrary element 
$$(\text{e}({\bf x}),\text{e}(A^{-1}({\bf x})) ,\;\text{e}(A^{-2}({\bf t})),\cdots, \text{e}(A^{-n}({\bf x})),\cdots, )\;\in\;L_w$$ 
will be an element of $\pi_0^{-1}(\{z\})$ for $z=\text{e}({\bf t}),\;{\bf t}\;\in\;[0,1)^d$ if and only if $\text{e}(({\bf x})=z$ which will occur if and only if ${\bf x}={\bf t}+{\bf k}$ for some ${\bf k}\in \Z^d,$ where ${\bf t}$ is the unique element in $[0,1)^d$   such that $\text{e}({\bf t})=z.$  We note that each ${\bf k}\in \mathbb Z^d$ gives rise to a different element of $E_w.$ Thus $\{w({\bf t}+{\bf k}):\;{\bf k}\in\Z^d\}\;=\;L_w\cap\pi_0^{-1}(\{z\})$ and for almost all $z,\;\nu_z$ is supported on this countable set.

To give an example, suppose $d=1,$ and suppose ${\bf k}\in\;\N$ is a positive integer. Give $k$ its $N_1=N$-adic expansion, i.e. write 
${\bf k}\;=\;\sum_{i=0}^{j_k-1}{\bf a}_iN^{i}$ for a suitable choice of digits ${\bf a}_i\in\{0,1,\cdots,N-1\}.$
Then for any $n\in N$ with $n\geq j_k,$ the number 
$$\text{e}(\frac{{\bf k}}{N^n})\;=\;\text{e}(\frac{\sum_{i=0}^{j_k-1}{\bf a}_iN^{i})}{N^n})=\;\text{e}(\frac{\sum_{i=0}^{j_k-1}{\bf a}_iN^{i})}{N^{n}}),$$
 so that considering the map $\Theta^{-1}:\CS_{\beta}\to \T\times \prod_{j=0}^{\infty}\{0,1,\cdots,N-1\},$ 
$$\Theta^{-1}(\{w({\bf t}+{\bf k})\})\;=\;\{(\text{e}({\bf t}), ({\bf a}_0,{\bf a}_1,\cdots, {\bf a}_{j_k-1},0,0,\cdots, )\}.$$  We deduce from the earlier portion of the proof that 
$$\nu_z(\{({\bf a}_0,{\bf a}_1,\cdots, {\bf a}_{j_k-1},0,0,\cdots, )\})\;=\;$$
$$=\;\lim_{n\to\infty}\frac{1}{N^n}\prod_{j=1}^{n-1}|m(\text{e}({\bf t}/N^j)\cdot \text{e}(\frac{\sum_{i=0}^{j-1}{\bf a}_iN^{i}}{N^j})|^2=\;\prod_{j=1}^{\infty}\frac{|m(\text{e}(({\bf t}+{\bf k})/N^j)|^2}{N}=\;|\widehat{\phi}({\bf t}+{\bf k})|^2.$$  Similarly, if $-{\bf k}$ is a negative integer, one can use the $N$-adic expansion of $-{\bf k}$ with only a finite number of digits not equal to $N-1$ to show that 
for $({\bf b}_0=N-{\bf a}_0,{\bf b}_1=N-1-{\bf a}_1,\cdots, {\bf b}_{j_k-1}=N-1-{\bf a}_{j_k-1},N-1,N-1, N-1,\cdots, )\in\;\prod_{j=0}^{\infty}\{0,1,\cdots,N-1\},$ we have 
$$\nu_z(\{({\bf b}_0,{\bf b}_1,\cdots, {\bf b}_{j_k-1},{\bf b}_{j_k}=N-1,N-1, N-1\cdots, )\})=\;|\widehat{\phi}({\bf t}-{\bf k})|^2.$$

It follows that setting $W_z= \;\pi_0^{-1}(\{z\}\cap L_w,$ we have that $W_z$ is countable, and since    
$$\nu_z(W_z)\;=\;\nu_z(\cup_{k\in\Z^d}\{(e^{\frac{2\pi i (t+k)}{N^n}})_{n=0}^{\infty}\})$$
$$=\;\sum_{k\in\Z^d}\nu_z(\{(\text{e}(A^{-n}({\bf t}+{\bf k}))_{n=0}^{\infty}\})\;=\;\sum_{k^d\in\Z}|\widehat{\phi}({\bf t}+{\bf k})|^2=\;1,$$
if we let $E_z$ consist of the set of those points in $W_z$ having positive measure, $E_z$ is both countable and made up of atoms with $\nu_z(E_z)=1.$  Since $\tau(\cup_{z\in \T^d}E_z)=\int_{\T^d}\nu_z(E_z)dz\;=1=\;\tau(\CS_{\beta}),$ it follows that we must have  $\nu_z(\pi_0^{-1}(\{z\}\backslash E_z)=0,$ for almost all $z\in \T^d.$
\end{proof}

\begin{example}
{\rm   Consider the classical Haar filter $m(z)=\frac{1+z+z^2}{\sqrt{3}}$ defined on $\T$ corresponding to dilation by $N=3.$  We have seen in the proposition above that the measure $\nu_z$  corresponding to the fiber 
$$\{z\}\times \prod_{j=0}^{\infty}\{0,1,2\}\;\subseteq\; \T\times \prod_{j=0}^{\infty}\{0,1,2\}$$ will have an atom at 
$({\bf a}_0,{\bf a}_1,{\bf a}_2,\cdots,)$ if and only if there exists $J$ such that we either have ${\bf a}_j=0 for every j>J,$ or we have ${\bf a}_j=2$ for every $j>J.$ This means that any atom $({\bf a}_0,{\bf a}_1,{\bf a}_2,\cdots,) \in \prod_{j=0}^{\infty}\{0,1,2\} \equiv \Sigma_3$ corresponds to the natural embedding of the integers in $\Sigma_3$.  Thus, if we have either ${\bf a}_j=0$ for all $j>J$ or ${\bf a}_j=2$ for all $j>J,$ and $z= e^{2\pi i t},$ the value of the fiber measure 
$\nu_z(\{({\bf a}_0,{\bf a}_1,{\bf a}_2,\cdots)\})$ will be equal to 
$|\widehat{\phi}({\bf t}+{\bf n})|^2,$ where $({\bf a}_0,{\bf a}_1,{\bf a}_2,\cdots)$ corresponds to the integer $n\in\mathbb Z.$
}
\end{example}
\begin{example}
{\rm Consider the non-classical filter $m(z)=\frac{1+z^2}{\sqrt 2}$ associated to the inflated Cantor set wavelet in \cite{dutjor42}.  We note that $m$ is Lipschitz continuous, and an easy application of the triangle inequality shows that $|m(z)|\leq \sqrt{2}<\sqrt{3}$ for all $z\in\T.$  It follows that for all $z\in \T,$ the measure $\nu_z$ never has atoms.   This is true as well to the non-classical filter $m(z,w)=\frac{1+z+w}{2}$ associated to the inflated Sierpinski gasket set wavelet in \cite{dmp}.  By the triangle inequality, we have $|m(z,w)|\leq \frac{3}{2}.$ The dilation on the inflated Sierpinski gasket set comes from the $2\times 2$ matrix $\left(\begin{array}{rr}
2&0\\
0&2\end{array}\right),$ and the square root of the determinant of this matrix is $2.$  Again, we have $\frac{|m(z,w)|}{2}\leq \frac{3}{4}<1$ for all $(z,w)\in \mathbb T^2.$   It follows that if one considers the measure $\tau$ on the inverse limit of $\T^2$ coming from the map $(z,w)\mapsto (z^2,w^2),$ the fiber measures on the generalized Sierpinski gasket sets corresponding to a fixed $(z,w)\in \T^2$ will not have atoms.}
\end{example}

\section{Finding generalized MSF wavelets in the $L^2$ spaces of solenoids} 

Let $m:\T^d\to\C$ be a Borel function with $\nu(m^{-1}(\{0\}))=0$  satisfying the filter equation $$\sum_{\{w: w^N=z\}}|m(w)|^2=N,\;a.e. z\in \T^d;$$ Here $N=|\mbox{det}(A)|,$ where $A$ is a diagonal expansive matrix. Let $\tau$ be the Borel measure on $\CS_{\beta}\;=\;\varprojlim (\T,z\mapsto z^N)$    constructed using the filter $m$ in Proposition \ref{tauexists}.  We recall from Corollary 5.8 of \cite{dut} and as modified by Corollary 6.5 of \cite{aijkln}, that if the filter $m$ gives rise to a multi-resolution analysis, either on an inflated fractal space $(X,\mu)$ or on $\R^d,$  there is a ``modified" Fourier transform ${\mathcal F}:\;L^2(X,\mu)\to L^2(\CS_{\beta},\tau),$ where $(X,\mu)$ represents the measure space carrying the original multiresolution analysis; recall that $X\subset \R^d$ and might be all of $\R^d$ with $\mu$ equal to Lebesgue measure, or might be an inflated fractal set with $\mu$ singular with respect to Lebesgue measure.

For ${\mathcal F}:\;L^2(X,\mu)\to L^2(\CS_{\beta},\tau),$ it was established in Theorem 6.4 of \cite{aijkln} that for any $f\in L^2(\CS_{\beta}, \tau),$ and every $v\in \Z^d,$
$${\mathcal F}\circ T_v \circ{\mathcal F}^{\ast}(f)((z_n)_{n=0}^{\infty})\;=\; z_0^v (f)((z_n)_{n=0}^{\infty}),\forall v\in\Z^d,$$
and 
$${\mathcal F}\circ D^{-1} \circ {\mathcal F}^{\ast}(f)((z_n)_{n=0}^{\infty})\;=\;m(z_0)\cdot f\circ \sigma^{-1}( (z_n)_{n=0}^{\infty}).$$
Here $\{T_v:\;v\in\;\Z^d\}$ are the unitary translation operators on $L^2(X,\mu)$ and $D$ is the unitary dilation operator.

Our aim in this section is to use our previous decomposition of probability measures on solenoids to set possible conditions under which $L^2(X,\mu)$ will contain within it a single unit vector $\psi$ such that $\{D^jT_v(\psi):\;v\in\Z^d,\;j\in \mathbb Z\}$ is an orthonormal basis for $L^2(X,\mu).$ It will be desirable that such a single wavelet $\psi$ should have as many properties of MSF (minimally supported frequency) wavelets in $L^2(\R^d)$ as possible. Setting $\widehat{\psi}={\mathcal F}(\psi),$ the ``single wavelet" condition is equivalent to finding a vector $\widehat{\psi}\in L^2(\CS_{\beta},\tau)$ such that 
$\{\widehat{D}^j\widehat{T_v}(\widehat{\psi}):\;\;j,k\in \mathbb Z\}$ is an orthonormal basis for $L^2(\CS_{\beta},\tau),$ where here $\widehat{D}= {\mathcal F}\circ D  \circ {\mathcal F}^{\ast}$ and $\widehat{T_v}\;=\;{\mathcal F}\circ T_v \circ {\mathcal F}^{\ast}.$ We model our study after the study of MSF wavelets in $L^2(\R^d).$

\begin{definition}
\label{MSF}
 Let $(X,\mu)$ be an inflated fractal set, and let $\psi$ be a unit vector in $L^2(X,\mu)$ that is a single orthonormal wavelet in $L^2(X,\mu)$ for translation by $\Z^d$ and dilation by the diagonal dilation matrix $A.$  We say that $\psi$ is a generalized MSF wavelet if $\widehat{\psi}$ is equal to $\lambda((z_n))\chi_E,$ for $E$ a Borel subset of $\CS_{\beta}$ such that $\tau(\sigma^j(E)\cap \sigma^k(E))=0$ for $j\not= k$ and $\tau(\CS_{\beta}\backslash (\cup_{j\in \Z} \sigma^j(E))=0,$ and some function $\lambda:\CS_{\beta}\to \C.$
\end{definition}

Definition \ref{MSF} is deceivingly simple.  For example, let $({\mathcal R}, \mu)$ be the inflated fractal set contained in $\R$ created from the Cantor set. It is an open question as to whether or not MSF-wavelets, or even single wavelets, occur in $L^2({\mathcal R}).$

\begin{remark}
{\rm 
Formula 6.7 in Theorem 6.6 of \cite{aijkln} shows us that even when $(X,\mu)=(\R^d,\text{Lebesgue}),$ and $m$ is a classical low-pass filter, and we compute ${\mathcal F}(\psi)$ for $\psi$ an ordinary MSF wavelet in $L^2(\mathbb R^d),$  we cannot expect that the function $\lambda$ will be a constant function.
}
\end{remark}

In the rest of this section, we investigate a few properties that $E\subset \CS_{\beta}$ and $\lambda:\CS_{\beta}\to \C$ would need to have in order that 
${\mathcal F}^{\ast}(\lambda\cdot \chi_E)$ be an MSF single wavelet in $L^2(X).$

First, given an MSF wavelet $\psi\in L^2(X,\mu),$ the family $\{T_v(\psi): v\in\Z^d\}$ is an orthonormal set in $L^2(X,\mu).$  It follows that 
$$\{ z_0^v(\lambda\cdot \chi_E):\; v\in\mathbb Z^d\}$$ is an orthonormal set in $L^2(\CS_{\beta},\tau),$ and from this we deduce that the family $$\{z^v \lambda\circ \Theta (z,(a_n))\chi_E\circ \Theta:\; v\in \Z^d\}$$ is an orthonormal set in 
$L^2(\T^d\times \mathbb Z_A^d, \widetilde{\tau}).$

  We now wish to apply the knowledge of $\widetilde{\tau}$ obtained in the previous section to this question.  Note that $\chi_E\circ \Theta = \;\chi_{\Theta^{-1}(E)},$
so our first aim will be to come up with conditions on a subset $E'\subset \T^d\times \mathbb Z_A^d$ and (by abuse of notation) $\lambda:\T^d\times \mathbb Z_A^d \to \C$ that will guarantee $\{ z^v \lambda \cdot (\chi_E'):\; v\in\mathbb Z^d\}$ is an orthonormal set in $L^2(\T^d\times \mathbb Z_A^d, \widetilde{\tau}).$  The following is just a restatement of the notion of orthonormality.
\begin{lemma}
\label{orthotrans}
Let $E'$ be a Borel subset of $\T^d\times \mathbb Z_A^d,$ and let $\lambda:\T^d\times \mathbb Z_A^d\to \C$ be a Borel function.  Define $h:\T^d\to \C$ by 
$$h(z)\;=\;\int_{E'_z}|\lambda(z,({\bf a}_n))|^2d\nu_z.$$
Then $\{ z^v \lambda\cdot\chi_{E'}):\; v\in\mathbb Z^d\}$ is an orthonormal set in $L^2(\T^d\times \mathbb Z_A^d, \widetilde{\tau})$ if and only if $h(z)\;\equiv 1$ as an element of $L^2(\T^d).$
\end{lemma}
\begin{proof}
 If $h(z)\;\equiv 1$ as an element of $L^2(\T^d),$ then by the theory of Fourier series, we know that
$\int_{\T^d}h(z)dz\;=\; 1,\;\text{a.e. } z\in \T^d$ and for every $v\in \mathbb Z^d\backslash \{0\},\;\int_{\T^d}z^v h(z)dz=0.$
For $v,v'\in\Z^d,$ the inner product of $z^v \lambda\cdot\chi_{E'}$ and $z^{v'} \lambda\cdot\chi_{E'},$ which we can write as  
$$\langle z^v \lambda \chi_{E'}, z^{v'}\lambda \chi_{E'}\rangle,$$
is exactly $$\int_{\T^d \times \mathbb Z_A^d}z^v \lambda\cdot\chi_{E'}\overline{z^{v'} \lambda\cdot\chi_{E'}} d\widetilde{\tau}$$
$$=\;\int_{\T^d}z^{v-v'}\left[\int_{E'_z} |\lambda(z,({\bf a}_n))|^2d\nu_z\right]dz$$
$$=\;\int_{\T^d}z^{v-v'}h(z)dz.$$
So if $h(z)\equiv 1,$ the family $\{z^v\lambda\cdot\chi_{E'}\}$ is orthonormal in $L^2(\T^d\times \mathbb Z_A^d.$  Conversely, if $\{z^v\lambda\cdot\chi_{E'}\}$ is orthonormal in $L^2(\T^d\times \mathbb Z_A^d,$ we must have $h(z)=1$ in $L^2(\T^d).$ 
\end{proof}

From this result, we obtain the following corollary, in the case where $\lambda$ depends only on the variable $z:$ 
\begin{corollary}
\label{wvsetsolen}
Let $E'$ be a Borel subset of $\T^d\times \mathbb Z_A^d,$ and $\lambda:\T^d\to \C$ is a Borel function.  Then $\{ z^v \lambda(z)\chi_{E'}):\; v\in\mathbb Z^d\}$ is an orthonormal set in $L^2(\T^d\times \mathbb Z_A^d, \widetilde{\tau})$ if and only if $\nu_z(E_z')=\frac{1}{|\lambda(z)|^2}\;\text{a.e.}\;z\; \in \;\T^d.$
\end{corollary}
\begin{proof}
Assume that $\{ z^v \lambda(z)\chi_{E'}):\; v\in\mathbb Z^d\}$ is an orthonormal set in $L^2(\T^d\times \prod_{j=0}^{\infty}[\mathbb Z_A]_j, \widetilde{\tau}).$ We know from Lemma \ref{orthotrans} that $h(z)\;=\;|\lambda(z)|^2\nu_z( E_z') =1 \; \text{a.e.}\;z \in \;\T^d.$
But this implies that $\nu_z(E_z')=\frac{1}{|\lambda(z)|^2}\;\text{a.e.}\;z\; \in \;\T^d.$

Conversely, if $\nu_z(E_z')=\frac{1}{|\lambda(z)|^2}\;\text{a.e.}\;z\; \in \;\T^d,$ then $h(z)=1 \; \text{a.e.}\;z \in \;\T^d,$ so that by Lemma \ref{orthotrans} $\{ z^v \lambda(z)\chi_{E'}):\; v\in\mathbb Z^d\}$ is an orthonormal set in $L^2(\T^d\times \mathbb Z_A^d, \widetilde{\tau}).$
\end{proof}

Note now that if $E'\subset \T^d\times \mathbb Z_A^d$ satisfies the conditions of Corollary \ref{wvsetsolen}, it will be the case that 
$$\overline{\text{span}}\{ z^v \lambda(z,({\bf a}_n))\chi_{E'}):\; v\in\mathbb Z^d\}\;\subset\;L^2(E',\tilde{\tau}).$$
The following proposition gives necessary and sufficient conditions on $E'$ to have  
$$\overline{\text{span}}\{ z^v \lambda(z,({\bf a}_n))\chi_{E'}):\; v\in\mathbb Z^d\}\; = \;L^2(E',\tilde{\tau}).$$

\begin{proposition}
\label{eventougher}
Let $E'\subset \T^d\times \mathbb Z_A^d$ and $\lambda:\T^d\times \mathbb Z_A^d\to \C$ satisfy the conditions of Corollary \ref{wvsetsolen}.  Then 
$$\overline{\text{span}}\{ z^v \lambda\cdot (\chi_{E'}):\; v\in\mathbb Z^d\}\; = \;L^2(E',\widetilde{\tau})$$
if and only if upon setting $z=e({\bf t}),$ for every  $m\in \mathbb Z^d$ and $j\in \mathbb N\cup\{0\},$ 
\begin{multline*}
[e(A^{-j}({\bf t})]^m [e(\sum_{i=0}^{j-1}A^{i-j}({\bf a_i}))]^m \lambda(e({\bf t}),({\bf a}_n))\chi_{E'}((e({\bf t}), ({\bf a}_0,{\bf a}_1,\cdots, {\bf a}_n,\cdots))\\
\;\in\;\overline{\text{span}}\{ z^v \lambda\cdot \chi_{E'}:\; v\in\mathbb Z^d\}.
\end{multline*}
\end{proposition}
\begin{proof}  Suppose that $\overline{\text{span}}\{ z^v \lambda\cdot \chi_{E'} :\; v\in\mathbb Z^d\}\; = \;L^2(E',\tilde{\tau}).$  Then for every $g\in L^{\infty}(\T^d\times \prod_{j=0}^{\infty}\{0,1,\cdots, N-1\}_j,\widetilde{\tau}),\;g(z,({\bf a}_j)_{j=0}^{\infty})\cdot \lambda\cdot \chi_{E'}\in\;L^2(E',\widetilde{\tau})=\overline{\text{span}}\{ z^v \lambda(z)\chi_{E'}:\; v\in\mathbb Z^d\}.$ In particular, if we fix $m\in\mathbb Z^d$ and $j\in\mathbb N\cup\{0\},$ the function 
$$g_{m,j}(e({\bf t}), ({\bf a}_0,{\bf a }_1,\cdots, {\bf a}_n,\cdots,))\;=\;[e(A^{-j}({\bf t}))]^m [e(\sum_{i=0}^{j-1}A^{i-j}({\bf a_i}))]^m\in\;L^{\infty}(\T^d\times \mathbb Z_A^d,\widetilde{\tau}),$$ so that $[e(A^{-j}({\bf t}))]^m [e(\sum_{i=0}^{j-1}A^{i-j}({\bf a_i}))]^m\cdot\lambda\cdot\chi_{E'}\in\;\overline{\text{span}}\{ z^v \lambda\cdot\chi_{E'}:\; v\in\mathbb Z^d\}.$ 

Conversely, suppose that $z=e({\bf t}).$  For every $m\in\mathbb Z^d$ and $j\in\mathbb N\cup\{0\},$ we have that 
\begin{multline*}
[e(A^{-j}({\bf t}))]^m [e(\sum_{i=0}^{j-1}A^{i-j}({\bf a_i}))]^m \lambda\cdot\chi_{E'}((e({\bf t}), ({\bf a}_0,{\bf a}_1,\cdots, {\bf a}_n,\cdots))\\
\in\;\overline{\text{span}}\{ z^v \lambda\cdot \chi_{E'}:\; v\in\mathbb Z^d\}.
\end{multline*}
  We now consider 
$\chi=A^{-j}(m)$ as an element of $\widehat{\CS_{\beta}}.$  For each $f\in L^2(\CS_{\beta},\tau),$
$$\widehat{T_{\chi}}(f)((z_n)_{n=0}^{\infty})\;=\;(z_j)^m\cdot f((z_n)_{n=0}^{\infty}).$$  Let $E\;=\;\Theta(E').$ Then to say that 
$$[e(A^{-j}({\bf t}))]^m [e(\sum_{i=0}^{j-1}A^{i-j}({\bf a_i}))]^m\lambda\cdot \chi_{E'}((e({\bf t}), ({\bf a}_0,{\bf a}_1,\cdots, {\bf a}_n,\cdots))\;\in\;\overline{\text{span}}\{ z^v \lambda\cdot \chi_{E'}:\; v\in\mathbb Z^d\}$$ is equivalent to saying that 
$\widehat{T_{\chi}}(\lambda \cdot \chi_{E}((z_n)_{n=0}^{\infty})\in \overline{\text{span}}\{ \widehat{T_v}(\lambda\cdot\chi_{E}):\; v\in\mathbb Z^d\},\forall \chi\in \widehat{\CS_{\beta}}.$
But if $\widehat{T_{\chi}}(\lambda\cdot\chi_{E})\;\in\;\overline{\text{span}}\{ \widehat{T_v}(\lambda\cdot\chi_{E}):\; v\in\mathbb Z^d\},\forall \chi\in \widehat{\CS_{\beta}},$ it follows that for every finite linear combination of characters $p$ in $\widehat{\CS_{\beta}},\;p((z_n)_{n=0}^{\infty})\;=\;\sum_{i=1}^K\;a_i\chi_i,$ 
$M_p(\lambda \cdot \chi_E)\;=\;[\sum_{i=1}^Na_i\widehat{T_{\chi_i}}](\lambda \cdot \chi_E)\in\;\overline{\text{span}}\{ \widehat{T_v}(\lambda\cdot\chi_{E}):\; v\in\mathbb Z^d\}.$  By the Stone-Weierstrass Theorem, $\{p\in\;C(\CS_{\beta}): p \;\;\text{a polynomial in}\;\widehat{\CS_{\beta}}\}$ is dense in $C(\CS_{\beta})$ in the uniform norm.  It follows that for every $h\in C(\CS_{\beta}),\;M_h(\lambda \cdot \chi_E)\;=\;h\cdot \lambda \cdot \chi_E)\in\;\overline{\text{span}}\{ \widehat{T_v}(\lambda\cdot\chi_{E}):\; v\in\mathbb Z^d\}.$  Since $\{M_h:\;h\in C(\CS_{\beta})\}$ is dense in $\{M_g: g\in L^{\infty}(\CS_{\beta},\tau)\}$ in the weak operator topology, we obtain that 
$\{M_g(\lambda \cdot \chi_E)=g\cdot \lambda \cdot \chi_E: g\in L^{\infty}(\CS_{\beta},\tau)\}$ is contained in $\overline{\text{span}}\{ \widehat{T_k}(\lambda\cdot\chi_{E}):\; k\in\mathbb Z\}.$ Thus 
$L^2(E,\tau)\;=\;\overline{\{M_g(\lambda \cdot \chi_E)=g\cdot \lambda \cdot \chi_E: g\in L^{\infty}(\CS_{\beta},\tau)\}}$ is contained in $\overline{\text{span}}\{ \widehat{T_v}(\lambda\cdot\chi_{E}):\; v\in\mathbb Z^d\}.$ 
Since it is obvious that $\overline{\text{span}}\{ \widehat{T_v}(\lambda\cdot\chi_{E}):\; v\in\mathbb Z^d\}\subset L^2(E,\tau),$ we see that 
$$\overline{\text{span}}\{ \widehat{T_v}(\lambda\cdot\chi_{E}):\; v\in\mathbb Z^d\}\;=\; L^2(E,\tau).$$  By our measure-theoretic isomorphism between $(\CS_{\beta},\tau)$ and $(\T^d\times \mathbb Z_A^d, \widehat{\tau}),$ this implies that 
$$\overline{\text{span}}\{ z^v \lambda\circ \Theta \cdot \chi_{E'}:\; v\in\mathbb Z^d\}\;=\;L^2(E',\widetilde{\tau}),$$
as desired.
\end{proof}

Corollary \ref{wvsetsolen} and Proposition \ref{eventougher} immediately give us the following result.
\begin{corollary}
\label{corMSFtotal}
Let $E'\subset \T^d\times \mathbb Z_A^d$ and suppose that $\lambda:\T^d\times \mathbb Z_A^d \to \C$ is a Borel function. Then $\lambda\cdot \chi_{E'}$ is a generalized MSF wavelet in $L^2(\T\times \mathbb Z_A^d,\widetilde{\tau})$ if and only if the following three conditions are satisfied:
\begin{enumerate}
\item The sets $\{\widetilde{\sigma}^j(E'): j\in\mathbb Z\}$ are pairwise disjoint up to sets of $\widetilde{\tau}$ and $\widetilde{\tau}(\T^d\times \mathbb Z_A^d\backslash \cup_{j\in\mathbb Z}\widetilde{\sigma}^j(E'))=0,$
\item $\int_{E'_z}|\lambda(z,({\bf a}_n))|^2d\nu_z=1,\;\text{a.e} z\in \T.$
\item $L^2(E',\widetilde{\tau})\;=\;\overline{\text{span}}\{ z^v \lambda\cdot \chi_{E'}:\; v\in\mathbb Z^d\}.$
\end{enumerate}
\end{corollary}

We can rewrite condition (3) in Corollary \ref{corMSFtotal}
to obtain the following characterization of generalized MSF
wavelets. 

\begin{theorem}\label{MSFpossible} Let $E'\subset
\T^d\times \mathbb Z_A^d$ and suppose
that $\lambda:\T^d\times \mathbb Z_A^d\to \C$ is a Borel function.  Then $\lambda\cdot
\chi_{E'}$ is a generalized MSF wavelet in $L^2(\T^d\times
\mathbb Z_A^d,\widetilde{\tau})$
if and only if the following three conditions are satisfied:
 \begin{enumerate} \item [(1)]The sets $\{\widetilde{\sigma}^j(E'):
j\in\mathbb Z\}$ are pairwise disjoint up to sets of $\widetilde{\tau}$
and $\widetilde{\tau}(\T^d\times \mathbb Z_A^d\backslash \cup_{j\in\mathbb Z}\widetilde{\sigma}^j(E'))=0,$
 \item [(2)]$\int_{E'_z}|\lambda(z,(a_n))|^2d\nu_z=1,\;\text{a.e. }
z\in \T^d.$ \item [($3'$)]For almost all $z\in\mathbb T,$
$E'_z\;=\;\{({\bf a}_j)_{j=0}^{\infty}: (z,({\bf a}_j)_{j=0}^{\infty})\in
E'\}$ has a single atom with respect to $\nu_z,$ that is, for almost
all $z,$ there exists $({\bf a}_j)_{j=0}^{\infty}\in\; E'_z$ such that
$\nu_z(\{({\bf a}_j)_{j=0}^{\infty}\})=\;\nu_z(E'_z).$
 \end{enumerate}
\end{theorem} 

\begin{proof} Assume that $\lambda \cdot \chi_{E'}$ is
a generalized MSF wavelet. By Corollary \ref{corMSFtotal}, we can assume
that conditions (1), (2)  and (3) hold for $E'.$  This setting is a
special case of a theorem of G. W. Mackey (Theorem 2.11 of \cite{mack})
about a direct integral of Hilbert spaces relative to a measure
that has a decomposition over a quotient space.  Mackey's theorem says that the
original direct integral is naturally isomorphic to the direct integral
over the quotient space of the direct integrals over the fibers.
In our case, the original Hilbert spaces are all equal to $\mathbb C$,
and the measure $\widetilde{\tau}$ is decomposed over $\mathbb T^d$ relative to
the Haar measure on $\mathbb T$, with fiber measures $\nu_z$ as in 
Proposition 2.4.  For each $z$ let $\H_z$ be $L^2(E',\nu_z).$  Then
Mackey's theorem assures us that the spaces $H_z$ form an integrable
bundle and that $$L^2(\widetilde{\tau}) \cong \int^\oplus H_z dz,$$ in such a way
that the class of a bounded Borel function $f$ in $L^2(\tau)$ corresponds to the
class of the section of the bundle over $\mathbb T^d$ whose value at each
$z$ is the class of $f$ in $H_z$.  
This works because the bounded Borel functions are square-integrable relative 
to every finite measure, providing a dense subspace
 in
every one of the Hilbert spaces involved.

In particular, the function
$\chi_{E'}$ corresponds in this particular way to a section $\psi$ of
the Hilbert bundle over $\mathbb T^d$ whose fibers are the spaces $H_z$. 
Multiplying $\psi$ by Laurent polynomials in $z$ times $\lambda(z)$ gives
a subspace of the sections of the bundle that is $L^2$-dense because
the corresponding functions on $E'$ are dense.  Hence, for almost every
$z$ the one-dimensional space spanned by $\psi(z)$ is dense in $H_z$.
Thus almost every $H_z$ is one dimensional, which implies that $\nu_z$
is a point mass for almost every $z$.
 
Conversely, if (1) and (2) hold
and almost every $\nu_z$ is a point mass, we have that almost every $H_z$ is one dimensional.
Thus the space of multiples of $\psi$ by Laurent polynomials times
$\lambda$ is dense in the space of $L^2$ sections. This implies that
the same multiples of $\chi_{E'}$ are dense in $L^2(\widetilde{\tau})$.
Thus, condition (3) holds, and by Corollary \ref{corMSFtotal} we have that
$\lambda\cdot
\chi_{E'}$ is a generalized MSF wavelet.
 \end{proof}

As an immediate consequence of the above theorem, we see that the MRA corresponding to the inflated Cantor fractal set and the inflated Sierpinski fractal set cannot have generalized MSF single wavelets:

\begin{corollary}
\label{nofractalMSF}
Let  $L^2(X,\mu)$ be a measure space carrying a translation by $\Z^d$ and dilation by the diagonal dilation matrix $A,$ and suppose that $L^2(X,\mu)$ has a MRA corresponding to a single scaling function $\phi$ and a finite wavelet family $\{\psi_i\}_{i=1}^{N-1}\subset L^2(X,\mu),$ with corresponding {\bf continuous}``low-pass" filter function $m:\T^d\to \C.$   If  $|m(z)|<\sqrt{N},\;\forall z\in\mathbb T^d,$  then $L^2(X,\mu)$ does not have a MSF wavelet corresponding to the translation and dilation operators. 
\end{corollary}
\begin{proof}
Since $|m(z)|<\sqrt{N},\;\forall\;z\;\in\;\mathbb T^d,$ the fiber measures $\nu_z$ will never have atoms, by Lemma \ref{lemmaatoms}.
By Theorem \ref{MSFpossible}, we see that $L^2(X,\mu)$ will not have MSF wavelets.
\end{proof}

\begin{remark}
\label{fracsingwave}
{\rm
Corollary \ref{nofractalMSF} does not rule out the existence of a single wavelet in the inflated fractal spaces corresponding to the Cantor set or Sierpinski gasket.  It just states that if a single wavelet does exist, it cannot be a generalized MSF wavelet as defined above.
}
\end{remark}

\begin{remark}
\label{clswvset}
{\rm
Let us consider Proposition \ref{wvsetsolen} in the context of Proposition \ref{decompclas}, where we have a classical low-pass filter giving rise to a classical MRA.  Let $N=2$ and $m(z)=\frac{1+z}{\sqrt{2}},$ which is the filter corresponding to the Haar wavelet.
Recall that a standard computation shows that in this case, the Fourier transform of the scaling function $\phi$ is given by 
$$\widehat{\phi}(t)\;=\;\frac{e^{2\pi it}-1}{2\pi it}.$$ 
In this case, we consider the Shannon wavelet set $W=[-1,-\frac{1}{2})\cup [\frac{1}{2},1)\;\subset\;\R.$  
The winding line map from $\R$ to $\CS_2$ is given by 
$$w(t)\;=\;(e^{2\pi i t},e^{\frac{2\pi i t}{2}},\;e^{\frac{2\pi i t}{2^2}},\cdots, e^{\frac{2\pi i t}{2^n}},\cdots, ).$$
The image of $W$ under the winding line map is $E_{-}\cup E_{+},$ where 
$$E_-\;=\;\{(e^{2\pi i t},e^{\frac{2\pi i t}{2}},\;e^{\frac{2\pi i t}{2^2}},\cdots, e^{\frac{2\pi i t}{2^n}},\cdots, ):\;t\in[-1,-\frac{1}{2})\}$$
and 
$$E_+\;=\;\{(e^{2\pi i t},e^{\frac{2\pi i t}{2}},\;e^{\frac{2\pi i t}{2^2}},\cdots, e^{\frac{2\pi i t}{2^n}},\cdots, ):\;t\in[\frac{1}{2},1)\}.$$

We now construct $E'\;=\;\Theta^{-1}(E_{-}\cup E_{+})\subset \;\T\times \prod_{j=0}^{\infty}\{0,1\},$
where one can compute that 
$$\Theta^{-1}((z_n)_{n=0}^{\infty})\;=\;(z_0,(z_n\cdot [c(z_0)_n]^{-1})_{n=0}^{\infty}).$$

Note that for $t\in [-1,-\frac{1}{2}),$
$$c(e^{2\pi i t})\;=\;(e^{2\pi i t},e^{\frac{2\pi i (t+1)}{2}},\cdots, e^{\frac{2\pi i (t+1)}{2^n}},\cdots, ),$$
and for $t\in [\frac{1}{2},1),$
$$c(e^{2\pi i t})\;=\;(e^{2\pi i t},e^{\frac{2\pi i (t-1)}{2}},\cdots, e^{\frac{2\pi i (t-1)}{2^n}},\cdots, ),$$
It follows that for $t\in [-1,-\frac{1}{2}),$
$$\Theta^{-1}\circ c(e^{2\pi i t})\;= (e^{2\pi i t},e^{2\pi i -\frac{1}{2}},\cdots,e^{2\pi i -\frac{1}{2^n}},\cdots,),$$
and for $t\in [\frac{1}{2},1),$  
$$\Theta^{-1}\circ c(e^{2\pi i t})\;= (e^{2\pi i t},e^{2\pi i \frac{1}{2}},\cdots,e^{2\pi i \frac{1}{2^n}},\cdots,).$$
This shows us that for $t\in [-1,-\frac{1}{2}),$
$$\Theta^{-1}\circ c(e^{2\pi i t})\;=\;(e^{2\pi i t},a_0=1=2-1,a_1=1,\cdots a_n=1,\cdots)$$
and for $t\in [\frac{1}{2},1),$
$$\Theta^{-1}\circ c(e^{2\pi i t})\;=\;(e^{2\pi i t},a_0=1,a_1=0,\cdots a_n=0,\cdots).$$
Thus $E'=\;E'_+\;\cup\;E'_-,$ for 
$$E'_-\;=\; \{(e^{2\pi i t},a_0=1=2-1,a_1=1,\cdots a_n=1,\cdots):\;t\in [-1,-\frac{1}{2})\}$$ and 
$$E'_+\;=\;\{(e^{2\pi i t},a_0=1,a_1=0,\cdots a_n=0,\cdots):\;t\in [\frac{1}{2},1)\}.$$

By Proposition \ref{decompclas}, for $t\in [-\frac{1}{2},0),$
$$\nu_{e^{2\pi i t}}(\pi_0^{-1}(e^{2\pi i t})\cap E')\;=\;|\widehat{\phi}(t+1)|^2,$$
and for $t\in [0,\frac{1}{2}),$
$$\nu_{e^{2\pi i t}}(\pi_0^{-1}(e^{2\pi i t})\cap E')\;=\;|\widehat{\phi}(t-1)|^2,$$

We deduce that defining the standard cross-section map
$\theta:\T\to \R$ such that $\theta(z)=t$ for $t\in [-\frac{1}{2},\frac{1}{2})$ with $z=e^{2\pi i t},$ and define 
$\lambda:\T\to \R$ by 
$$\lambda(z)=\;\frac{1}{\widehat{\phi}(\theta(z)+1)},\;\theta(z)\;\in\;[-\frac{1}{2},0),$$
and 
$$\lambda(z)=\;\frac{1}{\widehat{\phi}(\theta(z)-1)},\;\theta(z)\;\in\;[0,\frac{1}{2}),$$
the function $\lambda(z)\cdot \chi_{E'}$ is the image in $L^2(\T\times \prod_{j=0}^{\infty}\{0,1\}_j, \widetilde{\tau})$
of an MSF wavelet in $L^2(\R)$ for dilation by $2.$  This of course comes as no surprise, since we constructed both $\lambda(z)$ and $E'$ directly from the Shannon wavelet set.

}
\end{remark}


\section{The wavelet representation of the Baumslag-Solitar group, the measure $\tau,$ on $\CS_{\beta}$, and induced representations} 
\label{BS}

We now relate the existence of generalized MSF wavelets to certain properties of the associated representation of the 
Baumslag-Solitar group.  Recall that for $N\geq 2,$ the classical Baumslag-Solitar group $BS_N$ has two generators $a,\;b,$ satisfying the relations 
$$aba^{-1}=b^N.$$ 
The discrete group $BS_N$ can also be written as a semidirect product $\mathbb Q_N\rtimes_{\theta}\mathbb Z$ where $\theta(q)=Nq,$ for $q\in \mathbb Q_N,$ the $N$-adic rationals. We generalize the Baumslag-Solitar group to diagonal $d\times d$ dilation matrices $A$ with diagonal entries $N_1,\;N_2,\;\cdots,\;N_d$ all integers $\geq 2.$  Recall we consider the $A$-adic rationals $\mathbb Q_A\;=\;\cup_{j=0}^{\infty}A^{-j}[\mathbb Q^d],$ and defining $\theta:\mathbb Q_A\to  \mathbb Q_A$ by $\theta({\bf q})=A({\bf q}),$ we can parametrize $\mathbb Q_A\rtimes_{\theta}\mathbb Z$ as the set of pairs $\{({\bf q},m): {\bf q}\in \mathbb Q_A,\;m\in\mathbb Z\},$ where 
the group operation is given by 
$$({\bf q}_1,m_1)\;\cdot\;({\bf q}_2,m_2)\;=\;({\bf q}_1+A^{-m_1}({\bf q}_2), m_1+m_2).$$
Denote this group by $BS_A.$
If $X\subset \mathbb R^d$ is a subset invariant under integral translation and dilation by $A,$ and $\mu$ is a $\sigma$-finite Borel measure on $X$ that is invariant under translation by the integers and quasi-invariant via the constant $K>0$ under dilation by $A$, so that $\mu(A(S))\;=\;K\mu(S)$ for all Borel subsets $S$ of $X,$ there is a related representation of $BS_A$ on $L^2(X,\mu),$ called a {\bf wavelet representation} of $BS_A,$ defined on $L^2(X,\mu)$ by 
$$W(({\bf q},m))(f)(x)\;=\;(\sqrt{N})^mf(A^m({\bf x}-{\bf q})),\;f\in\;L^2(X,\mu).$$
 We remark that for $d=1,$ the generator $a\in BS_N$ corresponds to $(1,0)$ and the generator $b$ corresponds to $(0,-1)\;\in \;\mathbb Q_N\rtimes_{\theta}\mathbb Z.$

Let $m$ be a quadrature mirror filter on $\mathbb T^d$ with respect to dilation by $A$ such that the Haar measure of $m^{-1}(\{0\})$ is equal to $0.$ Usually the filter $m$ will be a polynomial filter coming from the self-similarity relation satisfied by the generating subset of $X\subset\R^d.$ In earlier sections, we studied the structure of the measure $\tau$ on $\mathcal S_{\beta}$ such that the wavelet representation of $BS_A$ on $L^2(X,\mu)$ is equivalent to the representation on $L^2(\mathcal S_{\beta},\tau)$ 
defined by  
$$W(v)(f)(z_n)_{n=0}^{\infty})\;=\;\widehat{T_v}(f)((z_n)_{n=0}^{\infty})\;=\;(z_0)^v\cdot\; f((z_n)_{n=0}^{\infty}),\;f\in L^2(\mathcal S_{\beta},\tau),\;v\in\;\mathbb Z^d,$$
and
$$W(b)(f)((z_n)_{n=0}^{\infty})\;=\;\widehat{D}^{-1}(f)((z_n)_{n=0}^{\infty})\;=\;m(z_0)\cdot f\circ \sigma^{-1}((z_n)_{n=0}^{\infty}).$$
For simplicity of notation we have used $W$ to denote both the (equivalent) representations of $BS_Arks
$
on $L^2(X,\mu)$ and  $L^2(\mathcal S_{\beta},\tau).$ 

We now suppose that a wavelet representation of the above form, either on $L^2(\CS_{\beta},\tau)$ or on $L^2(X,\mu),$ has a {\bf single} orthonormal wavelet, i.e. suppose there exists $\psi\in\;L^2(X,\mu)$ such that 
$\{D^jT_v(\psi):\;j\;\in\;\mathbb Z,\;v\in\;\mathbb Z^d,\}$ is an orthonormal basis for $L^2(X,\mu).$  Our aim is to characterize when $\psi$ is a generalized MSF frequency wavelet in terms of properties of the associated wavelet representation.

The following theorem is related to a theorem from Eric Weber's 1999 CU Ph.D. thesis \cite{web}, which gave necessary and sufficient conditions for a single wavelet for dilation by $N$ in $L^2(\mathbb R)$ to be an MSF wavelet, in terms of the invariance of the wavelet subspace $W_0\subset L^2(\mathbb R)$ under translation by $\mathbb R.$  In our more general case, the space $X\;\subset \;\mathbb R^d$ does not necessary carry an action by translation by all of $\mathbb R^d,$ but only by the group $\mathbb Q_A.$ As before, ${\mathcal F}:L^2(X,\mu)\to L^2(\CS_{\beta},\tau)$ represents the generalized Fourier transform of Dutkay from the Hilbert space associated to the (possibly fractal) space $X$ to the Hilbert space associated to the solenoid.

\begin{theorem} \label{MSFsolenoid}
Let $g$ be an orthonormal wavelet in $L^2(\CS_{\beta},\tau)$ corresponding to dilation by $\widehat{D}$ and translation operators $\widehat{T_v},\;v\in\mathbb Z^d.$  Then, the following are equivalent:
\begin{enumerate}
\item[i)] ${\mathcal F}^{-1}(g)$ is a generalized MSF wavelet in the sense of Definition \ref{MSF},
\item[ii)] For each $j\in \mathbb Z,$ the set ${\mathcal W}_j=\overline{\text{span}}\{\widehat{D}^j\widehat{T_v}(g);\; v\in\mathbb Z^d\}$ is invariant under the operators $\{\widehat{T}_{{\bf q}}\}$ for all $A$-adic rational numbers ${\bf q}\in\mathbb Q_A.$
\end{enumerate}
\end{theorem}
\begin{proof}
i) implies ii):  Suppose that $\psi\;= \;{\mathcal F}^{-1}(g)$ is a generalized MSF wavelet in the sense of Definition \ref{MSF}.  This means that $g((z_n)_{n=0}^{\infty})\;=\; \lambda(z_0)\chi_E((z_n)_{n=0}^{\infty})$ where the sets $\{\sigma^j(E): j\in\mathbb Z\}$ are pairwise disjoint up to sets of $\tau$-measure $0$ and $\tau(\CS_{\beta}\backslash \cup_{j\in\mathbb Z}\sigma^j(E))=0,$ with $\lambda:\mathbb T^d\to \mathbb C$ satisfying $\int_{\mathbb T^d}|\lambda(z)|^2\nu_z(E_z)dz\;=\;1.$ We now show that ${\mathcal W}_0\;=\;\overline{\text{span}}\{\widehat{T_v}(g);\; v\in\mathbb Z^d\}$ is invariant under the operators $\{\widehat{T_{{\bf q}}}:\;{\bf q}\in \mathbb Q_A\}.$  We note that by definition, for ${\bf q}=A^{-j}(v),\;v\in\;\mathbb Z^d,\;j\in\;\mathbb N\cup \{0\},$
$$\widehat{T_{{\bf q}}}(f)((z_n)_{n=0}^{\infty})\;=\;\langle {\bf q},(z_n)_{n=0}^{\infty}\rangle f((z_n)_{n=0}^{\infty})\;=\;(z_j)^v f((z_n)_{n=0}^{\infty}),\;f\in L^2(\CS_{\beta},\tau),\;v\in\;\Z^d.$$
To say that $f\in {\mathcal W}_0$ is equivalent to saying that $f$ vanishes almost everywhere off of $E.$ But if $f$ vanishes almost everywhere off of $E,$ it is clear that $\widehat{T_{{\bf q}}}(f)= (z_j)^v\cdot f$ vanishes almost everywhere off of $E,$ so that $\widehat{T_{{\bf q}}}(f)\in {\mathcal W}_0.$ Thus ${\mathcal W}_0$ is invariant under $\widehat{T_{{\bf q}}},\;\forall\; {\bf q}\in \mathbb Q_A.$  In a similar fashion, assuming that i) holds, $f\in {\mathcal W}_j$ if and only if $f$ vanishes $\tau$ almost everywhere off of $\sigma^j(E).$ But if $f$ vanishes $\tau$ almost everywhere off of $\sigma^j(E),\;\widehat{T_{{\bf q}}}(f)= (z_j)^m\cdot f$ will vanish $\tau$-almost everywhere off of $\sigma^j(E).$  Thus ${\mathcal W}_j$ is invariant under $\widehat{T_{\alpha}},$ and since ${\bf q}\in \mathbb Q_A$ was chosen arbitrarily, we see that ${\mathcal W}_j$ is invariant under $\widehat{T_{{\bf q}}},\;\forall\; {\bf q}\in \mathbb Q_A.$

ii) implies i): Let $\psi \;=\;{\mathcal F}^{-1}(g)$ be an orthonormal wavelet in $L^2(X,\nu),$ so that $\{T^k(\psi):\;k\in \mathbb Z\}$ forms an orthonormal set.  Then $\{\widehat{T}^k(g):\;k\in \mathbb Z\}$ is an orthonormal set in $L^2(\CS_{\beta},\tau),$ and by hypothesis, ${\mathcal W}_0\;=\;\overline{\text{span}}\{\widehat{T_v}(g):\;v\in \mathbb Z^d\}$ is invariant under $\widehat{T_{{\bf q}}},\;\forall\; {\bf q}\in \mathbb Q_A.$  Let $E=\text{supp}(g).$  We want to show that $\overline{\text{span}}\{\widehat{T_v}(g):\;v\in \mathbb Z^d\}\;=\;L^2(E,\tau).$ By the formula giving the operators $\widehat{T_v},$ for $v\in \mathbb Z^d,$ it is certainly true that 
$$\overline{\text{span}}\{\widehat{T_v}(g):\;v\in \mathbb Z^d\}\;\subset\;L^2(E,\tau).$$  To show that $\overline{\text{span}}\{\widehat{T_v}(g):\;v\in \mathbb Z^d\}\;=\;L^2(E,\tau),$ it suffices to show that whenever $f\in L^2(E,\tau)$ is a compactly supported simple function with support $F$ lying in $E,$ then $f\in\overline{\text{span}}\{\widehat{T_v}(g):\;v\in \mathbb Z^d\}.$  Since  for any fixed $j\in\mathbb Z,$ we know that ${\mathcal W_j}$ is invariant under 
$\widehat{T_v}\;\forall v\in \mathbb Z^d,$  it follows that $\widehat{T_{{\bf q}}}(g)\in \overline{\text{span}}\{\widehat{T_v}(g):\;v\in \mathbb Z^d\},$ for every ${\bf q}\in \mathbb Q_A\;=\widehat{\CS_{\beta}}.$  Since finite linear combinations of characters are norm-dense in $C(\CS_{\beta}),$ we obtain that $p((z_n))\cdot g((z_n))\in \overline{\text{span}}\{\widehat{T_v}(g):\;v\in \mathbb Z^d\},$ for any continuous function $p$ defined on $\CS_{\beta}.$  From this we deduce that if $q\in L^{\infty}(\CS_{\beta},\tau)$ is an essentially bounded function defined on $\CS_{\beta},$ then $q\cdot g$ is in $\overline{\text{span}}\{\widehat{T_v}(g):\;v\in \mathbb Z^d\}.$ Let $u\in L^{\infty}(\CS_{\beta},\tau)$ be the unique function of modulus $1$ defined on $\CS_{\beta}$ such that $u\cdot g\;=\;|g|.$  It follows that $|g|\in \overline{\text{span}}\{\widehat{T_v}(g):\;v\in \mathbb Z^d\}.$

For $n=1,$ let $E_1=\{(z_n)\in E: |g((z_n))|>1\},$ for $n\geq 2,$ let $E_n\;=\;\{(z_n)\in E: \frac{1}{n-1} \geq |g((z_n))| > \frac{1}{n}\}$, and define $F_n = F \cap E_n$. Note that the sets $\{F_n\}$ are pairwise disjoint and their union is equal to $F.$ Let $K=\text{sup}|f((z_n))|.$  Fix $\epsilon>0,$ and find $N$ such that  $\tau(\cup_{n>N}F_n)<\frac{\epsilon}{K}.$ Note that $\frac{1}{g}$ is bounded on $\cup_{n \leq N} F_n,$ since $|\frac{1}{g}|=\frac{1}{|g|},$ which is strictly less than $N$ on $\cup_{n \leq N}.$ Let $f_N\;=\;\frac{f}{g}\cdot \chi_{\cup_{n \leq N} F_n}.$   Then since $f$ is bounded, $f_0$ is bounded as well, and by our earlier remarks, $f_N\cdot g$ is an element of $\overline{\text{span}}\{\widehat{T}^k(g):\;k\in \mathbb Z\}.$  One calculates $f_N\cdot g\;=\; f\cdot \chi_{\cup_{n \leq N} F_n}.$  We see that $f_N\cdot g$ is equal to $f$ on $\cup_{n \leq N} F_n,$ so that 
$$\|f_N\cdot g-f\|\;=\;\|f \cdot\chi_{\cup_{n > N} F_n}\|\leq K\cdot\|\chi_{\cup_{n > N} F_n}\|$$
$$=\;K\cdot (\tau(\cup_{n>N}F_n))<\;\epsilon.$$
Since $\overline{\text{span}}\{\widehat{T_v}(g):\;v\in \mathbb Z^d\}$ is closed in $L^2(\CS_{\beta},\tau),$ we obtain that $f\;\in\;\overline{\text{span}}\{\widehat{T_v}(g):\;v\in \mathbb Z^d\}.$   Since the set of all such $f$'s is dense in $L^2(E,\tau),$ we obtain that 
$L^2(E,\tau)\;\subset\;\overline{\text{span}}\{\widehat{T_v}(g):\;v\in \mathbb Z^d\}.$  It follows that 
$$L^2(E,\tau)\;=\;\overline{\text{span}}\{\widehat{T_v}(g):\;v\in \mathbb Z^d\},$$ and $g=\lambda\cdot \chi_E,$ for $\lambda\in L^2(\CS_{\beta},\tau)$ so that ${\mathcal F}^{-1}(g)$ is a generalized MSF wavelet, as desired.
\end{proof}

Next we will show that in the case where one can find a single wavelet $\psi\in L^2(X,\;\mu),$ whether or not $\psi$ can be classified as a generalized MSF wavelet is characterized by whether the wavelet subspaces determined by $\psi$ give the wavelet representation  the structure of an induced representation.   First, we consider the case where there exists a subset $C$ of $\mathcal S_{\beta}\;=\;\widehat{\mathbb Q_A}$ such that $\text{Ind}_{\mathbb Q_A}^{BS_A}\int^{\oplus}_{C}\gamma d\tau$ is equivalent to the wavelet representation $W.$  In this case, the Imprimitivity Theorem of G. Mackey shows that if $W$ is induced from a representation $U$ of $\mathbb Q_A$ on the Hilbert space ${\mathcal L},$ there must be a projection-valued measure from $BS_A/(\mathbb Q_A)$ to $L^2(\CS_{\beta}, \tau)$ such that 
$$W_{({\bf q},k)}^{-1}P_E W_{({\bf q},k)}\;=\;P_{({\bf q},k)^{-1}\cdot E}.$$
Since $BS_A/(\mathbb Q_A)\cong \mathbb Z,$ any projection valued measure on $BS_N$ amounts to a orthogonal decomposition of $L^2(\CS_{\beta},\tau)$ into infinitely many closed subspaces:
$$L^2(\CS_{\beta},\tau)\;=\;\oplus_{n\in\mathbb Z} {\mathcal W}_n.$$

In particular, setting $P_{\{n\}}=P_{{\mathcal W}_n}:=P_n,$ so that $P_n$ represents the orthogonal projection of $L^2(\CS_{\beta}, \tau)$ onto ${\mathcal W}_n,$ Mackey's Theorem gives the necessary conditions 
$$W_{(0,-n)} P_0 W_{(0,n)}\;=\;P_n,\;\forall n\in\mathbb Z,$$
and $$W_{(-{\bf q},0)}P_0W_{({\bf q},0)}\;=\; P_0,\;\forall {\bf q}\in \mathbb Q_A.$$
This is equivalent to the statement that the subspace ${\mathcal W}_0$ is invariant under the translations $T_{{\bf q}},\;\forall {\bf q}\in \mathbb Q_N,$ or, what is the same thing, the subspaces ${\mathcal W}_j$ are invariant under all integer translations. 

Since $W_{(0,-1)}^{-1} P_0 W_{(0,-1)}=P_1,$ and $\widehat{D}=W_{(0,-1)}^{-1},$
we obtain $\widehat{D}({\mathcal W}_0)\;=\;{\mathcal W}_1.$  It follows that if the representation of $BS_A$ above is induced from a representation U of the subgroup $\mathbb Q_A$ in such a way that there exists $g\in {\mathcal W}_0$ with $\{\widehat{T_v}(g): v\in\Z^d\}$ being an orthonormal basis for ${\mathcal W}_0,$  then $\{\widehat{D}^j\widehat{T_v}(g)\}$ will be an orthonormal basis for ${\mathcal W}_j,$ for each $j\in\mathbb Z.$ By the induced representation assumption, $L^2(\CS_{\beta}, \tau)\;=\;\oplus_{j\in \mathbb Z}{\mathcal W}_j.$ If $g$ exists, it would follow that $g$ would be an orthonormal wavelet in $L^2(\CS_{\beta}, \tau)$ for dilation by $\widehat{D}$ and translation by $\{\widehat{T_v}: v\in\mathbb Z^d.$  We will now show that if any wavelet of this type exists, it must be a generalized MSF wavelet; i.e., there will exist $E\subset \CS_{\beta}$ with $\{\sigma^j(E):\; j\in \mathbb Z\}$ pairwise disjoint up to sets of $\tau$-measure $0$ with $\tau(\CS_{\beta}\backslash \cup_{j\in\mathbb Z}\sigma^j(E))=0,$ and a function $\lambda:\mathbb T^d\to \mathbb C$ such that $\int_{\mathbb T^d}|\lambda(z)|^2 \nu_z(E_z)dz\;=\;1$ and $g((z_n)_{n=0}^{\infty})\;=\;\lambda(z_0)\chi_E((z_n)_{n=0}^{\infty}).$
 
\begin{theorem}
\label{thmmsfinduced}
Let $\psi\in L^2(X,\mu)$ be a single wavelet with  associated unitary dilation and translation operators $D$ and  $\{ T_v:\; \in\;\mathbb Z^d\}$, i.e. suppose that 
$\{D^jT_v(\psi):\;j\in\;\mathbb Z,\;v\;\in\;\mathbb Z^d\}$ is an orthonormal basis for $L^2(X,\mu).$  Let ${\mathcal F}:L^2(X,\mu)\to L^2(\CS_{\beta},\tau)$ be the generalized Fourier transform of Dutkay corresponding to a multiresolution analysis coming from a self-similar space $L^2(X,\mu).$  Then, the following are equivalent:
\begin{enumerate}
\item[i)] $\psi$ is a generalized MSF wavelet.
\item[ii)] The wavelet subspaces $W_j\;=\;\overline{\mbox{span}}\{D^jT_v(\psi)\}$ are the closed subspace corresponding to a system of imprimitivity $\{P_j:\;[j]\in BS_{A}/{\mathbb Q}_{A}\}.$ 
\end{enumerate}
\end{theorem}

\begin{proof}
Assume that condition i) holds.  
We know that the wavelet representation generated by $D$ and $\{T_v:\;v\in\mathbb Z^d\}$ on $L^2(X,\mu)$ is unitarily equivalent to the representation on $L^2(\CS_{\beta},\tau)$ generated by 
$\widehat{D}$ and $\{\widehat{T_v}:\; v\in\;\mathbb Z^d\}.$  As in the beginning of this section, we denote this latter representation by $W.$
Theorem \ref{MSFsolenoid}, condition i) implies that for each $j\in\mathbb Z,\;{\mathcal W_j}$ is  invariant under the operators $\{\widehat{T_v}:\;v\in\mathbb Z^d\}.$  
From the discussion given prior to the statement of the Theorem, in order that the wavelet representation $W:BS_A\;\to\;{\mathcal U}(L^2(\CS_{\beta},\tau))$ be induced from a representation of $\mathbb Q_A,$ we need to construct a projection- valued measure $\{P_n:\;n\in\;\mathbb Z\}$ onto $L^2(\CS_{\beta},\tau),$ such that for every $({\bf q},n)\in \mathbb Q_A\rtimes_{\theta}\mathbb Z$ and every $S\subset \mathbb Z\equiv \mathbb Q_A\rtimes_{\theta}\mathbb Z/ \mathbb Z,$
\begin{equation}
\label{imprim}
W_{({\bf q},-k)}P_0(f)\;=\;P_kW_{({\bf q},-k)}f,\;\forall k\in\;\mathbb Z,\;{\bf q}\in\mathbb Q_A.
\end{equation}

By condition i), ${\mathcal F}(\psi)=\lambda\cdot \chi_E,$ where $E\subset \CS_{\beta}$ is a Borel set such that the sets $\{\sigma^j(E):j\in\mathbb Z\}$ pairwise disjoint up to sets of $\tau$-measure $0$ and $\tau(\CS_{\beta}\backslash [\cup_{j\in \mathbb Z}\sigma^j(E)])=0.$    Setting ${\mathcal W}_j\;=\;L^2(\sigma^{-j}(E),\tau),$ we claim that defining $P_j\;=\;P_{{\mathcal W}_j},$ we obtain a projection-valued measure on $\mathbb Z$ satisfying the desired imprimitivity conditions.   First we note that ${\mathcal W}_j\perp {\mathcal W}_k$ for $j\not=k,$ so that the projections $\{P_j:\;j\in\mathbb Z\}$ are mutually orthogonal. Secondly, it's clear that $\oplus_{j\in\mathbb Z}L^2(\sigma^j(E),\tau)\;=\;L^2(\CS_{\beta},\tau).$  Therefore our definition provides a projection-valued measure from $\mathbb Z$ into projections on $L^2(\CS_{\beta},\tau).$   We now show that this projection-valued measure satisfies the requirements stated in Equation \ref{imprim}.

To verify Equation \ref{imprim}, it is enough to show that 
$$W({\bf q},k)P_{j}\;=\;P_{{\mathcal W}_{j-k}}W({\bf q},k),\;\forall ({\bf q},k)\in \mathbb Q_A\rtimes_{\theta}\mathbb Z,\;\forall j \in \mathbb Z.$$

Recall that given unitary operators $\widehat{D}$ and $\widehat{T_{{\bf q}}},\;{\bf q}\in \mathbb Q_A$ on $L^2(\CS_{\beta},\tau),$ as defined above,
$W({\bf q},n)$ is defined by 
$$W({\bf q},n);=\;\widehat{T_{{\bf q}}}\cdot \widehat{D}^{n}.$$
we obtain the wavelet representation of the Baumslag-Solitar group $BS_A.$

We calculate: 
$$W({\bf q},k)P_jf((z_n))\;=\;\langle {\bf q} (z_n)\rangle \cdot \widehat{D}^{k}f((z_n))\cdot \chi_{\sigma^{j}(E)}$$
$$=\;\langle {\bf q}, (z_n)\rangle \cdot \prod_{i=1}^k[\frac{1}{m(z_i)}]\cdot f\circ \sigma^k( (z_n)_{n=0}^{\infty})\chi_{\sigma^{j-k}(E)}.$$
On the other hand, 
$$P_{j-k}W({\bf q},k)f((z_n))\;=\;\chi_{\sigma^{j-k}(E)}\cdot \widehat{T_{{\bf q}}}\widehat{D}^{k}f((z_n))$$
$$=\;\langle {\bf q}, (z_n)\rangle \cdot \chi_{\sigma^{j-k}(E)} \cdot \prod_{i=1}^k[\frac{1}{m(z_i)}]\cdot f\circ \sigma^k( (z_n)_{n=0}^{\infty}).$$
We have established the equality 
$$W({\bf q},k)P_j\;=\;P_{j-k}W({\bf q},k)$$
and it follows that Equation \ref{imprim} is satisfied, so that $W$ is a representation that is induced from a representation on $\mathbb Q_A,$ with imprimitivity structure provided by the pairwise orthogonal wavelet subspaces. We thus have established condition (ii).

Assume now that condition (ii) holds.  Then Equation \ref{imprim} is satisfied with respect to the wavelet subspaces 
$\{\mathcal W_j:\; j\in\mathbb Z\}.$  In particular, Equation \ref{imprim} is satisfied with respect to $W_0=\overline{\text{span}}\{\widehat{T}^k({\mathcal F}(\psi):\; k\in\;\mathbb Z\}.$ so that $W({\bf q},0)P_0\;=\;P_0W({\bf q},0),\;\forall\;{\bf q}\;\in\;\mathbb Q_A,$ or, what is the same thing,
$W_0$ is invariant under the translation operators $T_{{\bf q}},\;\forall\;{\bf q}\;\in\;\mathbb Q_A.$  
 Note that
$$\widehat{D}^j\widehat{T_{\bf q}}{\widehat D}^{-j}=\widehat{T_{A^{-j}({\bf q})}},\;\forall\; j\in\;\mathbb Z,\;\text{and all}\;{\bf q}\in\mathbb Q_A,$$
so that $$\widehat{T_{A^{-j}({\bf q})}}\widehat{D}^j\;=\;\widehat{D}^j\widehat{T_{\bf q}},\;\forall\; j\in\;\mathbb Z,\;\text{and all}\;{\bf q}\in\mathbb Q_A.$$
Since ${\mathcal W}_j=\widehat{D}^j({\mathcal W}_0,)$  for any $h\in {\mathcal W}_j,$ we can find $f\in {\mathcal W}_0$ with $h\;=\;\widehat{D}^j(f).$  Then if ${\bf r}\in\mathbb Q_A,$
$$\widehat{T_{\bf r}}(h)\;=\;\widehat{T_{\bf r}}\widehat{D}^j(f)=\widehat{D}^j\widehat{T_{A^j({\bf r})}}(f).$$
Since we have shown $W_0$ is invariant under the translation operators $T_{\beta},\;\forall\;\beta\;\in\;\mathbb Q_N$, $\widehat{T_{A^j({\bf r})}}(f)\in\;{\mathcal W}_0.$  But then $\widehat{D}^j\widehat{T_{A^j({\bf r})}}(f)\in \;\widehat{D}^j({\mathcal W}_0)\;=\;{\mathcal W}_j.$
It follows that for each fixed $j\in\mathbb Z,\;{\mathcal W}_j$ is invariant under $\widehat{T_{\bf r}},\;\forall\; {\bf r}\in \mathbb Q_A.$ 
\end{proof}

Now suppose that either of the equivalent conditions of Theorem \ref{thmmsfinduced} are satisfied, and $E$ is the subset of  $(\CS_{\beta},\tau)$ serving as the candidate for a ``wavelet set".    For each $(z_n)\in E,$ by Pontryagin duality $(z_n)$ is a character on $\mathbb Q_A$ so that the pairing $\langle {\bf q}, (z_n)\rangle$ defines a one-dimensional unitary representation of $\mathbb Q_A.$  If we take the direct integral of these representations of $\mathbb Q_A,$ we obtain the direct integral representation 
$\int^{\oplus}_E (z_n) d\tau$ of $\mathbb Q_A$ on $L^2(E,\tau),$ given by the obvious formula 
$$L({\bf q}) f((z_n)_{n=0}^{\infty})\;=\;\langle {\bf q}, (z_n)\rangle f((z_n)_{n=0}^{\infty}).$$
Recall that the process of inducing and taking direct integrals commutes, so that 
$$\text{Ind}_{\mathbb Q_A}^{BS_A}( \int^{\oplus}_E (z_n) d\tau)\;=\;\int^{\oplus}_E[\text{Ind}_{\mathbb Q_A}^{BS_A}(z_n)_{n=0}^{\infty}] d\tau.$$  

A computation similar to that given in Theorem 2.1 of \cite{lpt} gives the following 

\begin{corollary} Let $g$ be an orthonormal wavelet in $L^2(\CS_{\beta},\tau)$ corresponding to dilation by $\widehat{D}$ and ``translation" by the operators $\{\widehat{T_v}:v\in\mathbb Z^d\},$ and suppose that $g$ corresponds to a generalized MSF wavelet, so that $g=\lambda\cdot \chi_E$ where the sets $\{\sigma^k(E): k\in\mathbb Z\}$ tile $\CS_{\beta}$ with respect to the measure $\tau.$ Then the wavelet representation $W$ of $BS_A$ on $L^2(\CS_{\beta},\tau)$ is unitarily equivalent the direct integral of induced representations 
$\int^{\oplus}_E[\text{Ind}_{\mathbb Q_A}^{BS_A}(z_n)_{n=0}^{\infty}] d\tau.$
\end{corollary}
\begin{proof}
The proof is very similar to that given in Theorem 2.1 of \cite{lpt}.  The most technical part is to construct a Hilbert space isomorphism between  $L^2(\CS_{\beta},\tau)$ and $L^2(E\times \mathbb Z, \tau\times \text{count})$ that intertwines the wavelet representation $W$ and the induced representation in the desired fashion.  We leave details to the reader.
\end{proof}

\begin{remark}
{\rm
Thus, a single wavelet $\psi\in L^2(X,\mu)$ can be used to construct orthogonal subspaces that regulate the induction of a wavelet representation from a representation of the normal subgroup $\mathbb Q_N$ if and only if $\psi$ is a generalized MSF wavelet. Moreover, the induced representation involved will be a direct integral of monomial representations, as in \cite{lpt}
}
\end{remark}



\end{document}